\documentclass[10pt]{amsart}

\usepackage{color,enumitem,geometry,indentfirst,latexsym,mathrsfs,mfirstuc,upgreek}
\usepackage[colorlinks,linkcolor=blue,urlcolor=blue]{hyperref}
\hypersetup{urlcolor=blue, citecolor=red}

\usepackage{enumitem} 
 \numberwithin{equation}{section}
\newtheorem{theorem}{Theorem}[section]

\newtheorem{lemma}[theorem]{Lemma}
\newtheorem{proposition}[theorem]{Proposition}

\newtheorem{question}[theorem]{Question}
\newtheorem{example}[theorem]{Example}
\newtheorem{remark}[theorem]{Remark}

\theoremstyle{definition}
\newtheorem{definition}[theorem]{Definition}

\newtheorem*{BA}{Basic assumption}  

\newtheorem*{TheoremA}{Theorem A}
\newtheorem*{TheoremB}{Theorem B}
\newtheorem*{TheoremC}{Theorem C}

\def\id{{\rm id}}

\title[Subsystems with shadowing property for $\mathbb{Z}^{k}$-actions] 
      {Subsystems with shadowing property for $\mathbb{Z}^{k}$-actions}
\author[Lin Wang, Xinsheng Wang and Yujun Zhu]{\scshape Lin Wang$^1$, Xinsheng Wang$^2$  and Yujun Zhu$^{3,*}$}



\email{linwang@sxufe.edu.cn}
\email{xswang@xmu.edu.cn}
\email{yjzhu@xmu.edu.cn}

\thanks{This work is supported by  NSFC (No: 11771118, 11801336).}
\thanks{$^*$ The corresponding author.}

\begin{document}

\maketitle

\date{}
\medskip

 {\footnotesize
   \centerline{1. School of Applied Mathematics}
   \centerline{Shanxi University of Finance and Economics, Taiyuan, 030006, P.R. China}
   \centerline{2,3. School of Mathematical Sciences}
   \centerline{ Xiamen University, Xiamen, 361005, P.R. China}
} 

\bigskip

\begin{abstract}
In this paper, subsystems with shadowing property for
$\mathbb{Z}^{k}$-actions are investigated. Let $\alpha$ be a continuous $\mathbb{Z}^{k}$-action on a compact metric space $X$. We introduce the notions of pseudo orbit and shadowing property for $\alpha$ along subsets, particularly subspaces, of $\mathbb{R}^{k}$. Combining with another important property ``expansiveness" for subsystems of $\alpha$ which was introduced and systematically investigated by Boyle and Lind in \cite{Boyle}, we show that if $\alpha$ has the shadowing property and is expansive along a subspace $V$ of $\mathbb{R}^{k}$, then so does for $\alpha$ along any subspace $W$ of $\mathbb{R}^{k}$ containing $V$. Let $\alpha$ be a smooth $\mathbb{Z}^{k}$-action on a closed Riemannian manifold $M$, $\mu$ an ergodic probability measure and $\Gamma$ the Oseledec set. We show that, under a basic assumption on the Lyapunov spectrum,  $\alpha$ has the shadowing property and is expansive  on $\Gamma$ along any subspace $V$ of $\mathbb{R}^{k}$ containing a regular vector; furthermore, $\alpha$ has the quasi-shadowing property on $\Gamma$ along any 1-dimensional subspace $V$ of $\mathbb{R}^{k}$ containing a first-type singular vector.  As an application, we also consider the 1-dimensional subsystems (i.e., flows) with shadowing property for the $\mathbb{R}^{k}$-action on the suspension manifold induced by $\alpha$.

\end{abstract}

\renewcommand{\sectionmark}[1]{}

\section{Introduction}

When studying a continuous $\mathbb{Z}^{k}$-action $\alpha$ on a compact metric space, one is naturally led to be interested not only in the ``overall'' dynamics of $\alpha$, but also in the dynamics of its ``subsystems'' defined by subgroups of $\mathbb{Z}^{k}$ and, more generally, the dynamics along subspaces of $\mathbb{R}^{k}$. First steps in this direction were made in a paper by Milnor \cite{Milnor}, where the concept of directional topological entropy was introduced and studied for a special class of continuous actions. To consider the continuity of the directional entropy  and other aspects in \cite{Nasu}, Boyle and Lind \cite{Boyle} provided a framework for studying the dynamics, especially the expansive subdynamics, of $\mathbb{Z}^k$.

The goal of this paper is to investigate the  shadowing property for, especially smooth, $\mathbb{Z}^{k}$-actions and to give the characterization of the subsystems which has the shadowing property. Combining the two multifaceted properties: ``shadowing property" and ``expansiveness", we will have some insight into the interaction between the ``overall" dynamics and the ``subdynamics" of a $\mathbb{Z}^{k}$-action.

It it well known that the property of shadowing describes the behavior of approximate trajectories (i.e., pseudo orbits) of a dynamical system. We say that a system has the shadowing property if there exists a real trajectory staying uniformly close to  any approximate trajectory. The shadowing property was first established by Anosov \cite{Anosov} and Bowen \cite{Bowen} for smooth systems on the hyperbolic sets, and now the shadowing theory has been developing on the powerful basis of the stability theory. For the development of the shadowing theory, please refer to the books \cite{Pilyugin}, \cite{Palmer} and \cite{Pilyugin3}.

For all we know, the shadowing theory for actions of groups more general than  $\mathbb{Z}$ was firstly considered by Pilyugin and  Tikhomirov \cite{Pilyugin1}.
One of their main results is that if there exists a 1-dimensional
subgroup of $\mathbb{Z}^{k}$ which has the shadowing property and is expansive, then so does for the ``overall'' action of $\mathbb{Z}^{k}$. Moreover, they also gave the conditions of the shadowing property for a classical linear $\mathbb{Z}^{k}$-action $\alpha$ on $\mathbb{C}^m$ generated by pairwise commuting matrices.
We can see further study of the shadowing theory for group actions from \cite{Koscielniak}, \cite{Kulczycki}, \cite{Osipov}, \cite{Pilyugin2}, etc.

As we mentioned above, our aim  is to investigate the  shadowing property for the subsystems of a $\mathbb{Z}^{k}$-action $\alpha$. Here, by the term ``subsystem" we  mean not only the action of a subgroup of $\mathbb{Z}^{k}$ but also the restriction of $\alpha$ to a subspace, even a subset,  of $\mathbb{R}^{k}$. Since a subset $F$ of $\mathbb{R}^k$ may be ``invisible"
within $\mathbb{Z}^k$, we will use the technique of thickening $F$ by a positive number $t$. This treatment is inspired by the work of Boyle and Lind \cite{Boyle} in which the expansiveness and continuity of entropy for subsystems of $\mathbb{Z}^{k}$-actions were systematically investigated.
We will investigate the shadowing property of $\alpha$ along a subspace $V$ of $\mathbb{R}^k$ inductively. We emphasize here that in order to study the shadowing property for a subsystem $V$ of $\alpha$, we must combine the expansivity and the shadowing property for certain carefully selected subsystems of $V^t$ for some $t$ together in the process of induction.

In Section 2, we will state some basic facts about expansive subdynamics from \cite{Boyle}, introduce the definitions of pseudo orbit, shadowing property and topologically Anosov of subsystems for a continuous $\mathbb{Z}^{k}$-action $\alpha$ on  a compact metric space $(X, d)$, and give some basic properties about them (Proposition \ref{BP1}, Proposition \ref{BP2} and Proposition \ref{BP3}). Particularly, we show in Proposition \ref{BP3} that if $\alpha$ has the shadowing property and is expansive along a subspace $V$ of $\mathbb{R}^{k}$, then so does for $\alpha$ along any subspace $W$ of $\mathbb{R}^{k}$ containing $V$. It is a generalization of Theorem 1 of \cite{Pilyugin1} to a broad case. The idea in the proof of this proposition is quite useful for our further discussion. In Section 3, we consider the shadowing property of the subsystems for several classical $\mathbb{Z}^{k}$-action, such as symbolic systems and the ``linear" systems on torus.

Section 4 is devoted to our main results on shadowing property for smooth $\mathbb{Z}^{k}$-actions. Let $M$ be an $m$-dimensional closed Riemannian manifold. We denote by $\langle\!\langle\cdot,\cdot\rangle\!\rangle$,  $\|\cdot\|$ and $d(\cdot,\cdot)$ the inner product, the norm on the tangent spaces and the metric on $M$ induced by the Riemannian metric, respectively. Let $\alpha$ be a $C^r, r\ge 1,$ $\mathbb{Z}^k$-action on $M$ and $\mu$ an ergodic invariant
measure. According to Multiplicative Ergodic Theorem for $\alpha$ (see \cite{Hu}, \cite{Kalinin} and \cite{Brown} for example) the Lyapunov decompositions for individual elements of $\alpha$ have a common
refinement $T_{\Gamma}M=\bigoplus_{j=1}^{s} E_j$, where $\Gamma$ is the Oseledec set with full $\mu$-measure. Let $\{(\lambda_{i,j},m_j): 1\le i\le k, 1\le j\le s\}$ be the spectrum of $\alpha$ on $\Gamma$, where $m_j=\dim E_j$ and $\lambda_{i,j}$ is the Lyapunov exponent of the $i$th generator of $\alpha$ with respect to  $E_j$ for $x\in\Gamma$ defined by
\begin{equation}\label{ULY}
\lim_{n\longrightarrow \pm \infty}\frac{1}{n}\log\|Df_i^n(x)v\|=\lambda_{i,j},\;\;0\neq v\in E_j(x).
\end{equation}
When considering the shadowing property in the classical theory of diffeomorphisms, one often require that the system have some  hyperbolicity in certain uniform senses (see \cite{Palmer}, \cite{Pilyugin} and  \cite{Pilyugin3} for example). So we will consider the shadowing property for $\alpha$ under the following condition on the uniformity of  the convergence in the above limits.

\begin{BA}\label{Basicassumption}
 Let $\alpha$ be a $C^r,r\ge 1,$ $\mathbb{Z}^k$-action  on $M$ and $\mu$ an $\alpha$-ergodic measure.
We assume that the limits in (\ref{ULY})
are \emph{uniform} in the following sense:
for any $a>0$, there exists $N=N(a)>0$ such that when $|n|\ge N$,
\begin{equation}\label{Uniformity}
\lambda_{i,j}-a\le\frac{1}{n}\log\|Df_i^n(x)v\|\le\lambda_{i,j}+a,
\end{equation}
for any $v$ in the unit sphere in $E_j(x)$ , $1\le i\le k,  1\le j\le s, x\in\Gamma$.
\end{BA}

We say a nonzero vector $v=(v_1,\cdots,v_k)\in \mathbb{R}^k$ is \emph{regular} for  $\alpha$ if
$\sum_{i=1}^kv_i\lambda_{i,j}\neq 0\;\;\text{ for any }1\le j\le s$, otherwise we call it is \emph{singular}. Let $\mathbf{v}\in\mathbb{R}^k$ be a singular vector, we say it is a \emph{first-type} singular vector if there exist at least two indexes $j,j'\in \{1,\cdots,s\}$ such that
\begin{equation}\label{1singular}
\sum_{i=1}^kv_i\lambda_{i,j}=0\text{ and }\sum_{i=1}^kv_i\lambda_{i,j'}\neq 0.
\end{equation}

For $1\le i\le k$ and $1\le j\le s$, let $\mathbb{G}_i$ be the Grassmann manifold consists
of $i$-dimensional subspaces of $\mathbb{R}^k$ and $$\mathbb{A}_j(\alpha, \Gamma)=\{V\in \mathbb{G}_j: V \text{ is topologically Anosov  on } \Gamma \text{ for }\alpha\}.$$ The main results of this paper are as follows (one can see the precise definitions and notations in the next sections).

\begin{TheoremA}\label{ThmA}
Let $\alpha$ be a $C^r,r\ge 1,$ $\mathbb{Z}^k$-action  on $M$ and $\mu$ an $\alpha$-ergodic measure which satisfies the basic assumption. Let  $\mathbf{v}=(v_1,\cdots,v_k)\in\mathbb{R}^k$. If $\mathbf{v}$ is regular,
then $L_\mathbf{v}\in\mathbb{A}_1(\alpha, \Gamma)$. Hence for each $1\le i\le k$,
$$
\mathbb{A}_i(\alpha, \Gamma)=\{V\in \mathbb{G}_i : V \text{ contains a regular vector} \}.
$$
\end{TheoremA}

\begin{TheoremB}\label{ThmB}
Let $\alpha$ be a $C^r,r\ge 1,$ $\mathbb{Z}^k$-action  on $M$ and $\mu$ an $\alpha$-ergodic measure which satisfies the basic assumption. Let  $\mathbf{v}=(v_1,\cdots,v_k)\in\mathbb{R}^k$. If $\mathbf{v}$ is a first-type singular vector,
then $L_\mathbf{v}$ has the quasi-shadowing property for $\alpha$ on $\Gamma$.
\end{TheoremB}

As an application, we consider, in Section 5, the shadowing property for the $\mathbb{R}^{k}$-action $\tilde{\alpha}$ (the suspension) on a quotient space of $\mathbb{R}^{k}\times M$ induced by $\alpha$. In this case, every subsystem of $\tilde{\alpha}$ along a 1-dimensional subspace $V$ of $\mathbb{R}^k$ induces a flow.

\begin{TheoremC}\label{ThmC}
Let $\alpha$ be a $C^r,r\ge 1,$ $\mathbb{Z}^k$-action  on $M$ and $\mu$ an $\alpha$-ergodic measure which satisfies the basic assumption, and $\tilde{\alpha}$ be the suspension
of $\alpha$. Let  $\mathbf{v}=(v_1,\cdots,v_k)\in\mathbb{R}^k$.

(1) If $\mathbf{v}$ is regular for $\alpha$, then $\tilde{\alpha}$ has the shadowing property along $L_\mathbf{v}$.

(2) If $\mathbf{v}$ is a first-type singular vector for $\alpha$, then $\tilde{\alpha}$ has the quasi-shadowing property along $L_\mathbf{v}$.
\end{TheoremC}

We point out here that it is Pilyugin and  Tikhomirov's work \cite{Pilyugin1}  making us pay attention to this topic and giving us the initial idea. Combining the thickening technique of Boyle and Lind (\cite{Boyle}) and the method in studying shadowing property, especially quasi-shadowing property (\cite{Hu1}),  we investigate the shadowing properties for subsystems of $\mathbb{Z}^{k}$-actions.

\begin{remark}
One can see that the condition (\ref{ULY}) in the basic assumption is strong.  Under such an condition the smooth $\mathbb{Z}^k$-action $\alpha$ is uniformly hyperbolic along the regular direction (resp. partially hyperbolic along the first-type singular direction). Therefore, we can achieve one of the main goals of this paper, that is, to give the description of the subspaces of $\mathbb{R}^{k}$ along which $\alpha$ is topologically Anosov (i.e., $\alpha$ has the shadowing property and is expansive) on $\Gamma$. Once we only consider the dynamics of  the ``entire" but not the ``subsystems" of $\mathbb{Z}^k$-action $\alpha$, we can use the classical Pesin theory to consider the shadowing properties under a weaker condition. We mentioned that Pan, Zhang and Zhou \cite{Pan} did such work for $\mathbb{Z}^d$-actions which has a nonuniformly hyperbolic or partially hyperbolic generator.
\end{remark}

\section{Definitions, notations and basic properties}

Let $(X, d)$ be a compact metric space. Let $\mathbb{Z}^{k}, k>1$, be the additive group and Homeo$(X, X)$ the group of homeomorphisms on $X$. A $\mathbb{Z}^{k}$-\emph{action} $\alpha$ on $X$ is a homomorphism from $\mathbb{Z}^{k}$ to Homeo$(X, X)$. For $\mathbf{n}\in \mathbb{Z}^{k}$, we denote the corresponding homeomorphism
by $\alpha^\mathbf{n}$, so that $\alpha^{\mathbf{n}+\mathbf{m}}=\alpha^{\mathbf{n}}\circ \alpha^{\mathbf{m}}$ and $\alpha^\mathbf{0}$ is the identity on $X$.  We denote the collection of generators of $\alpha$ by
\begin{equation}\label{1.4}
  \{f_{i}=\alpha^{\mathbf{e}_{i}}:1\leq i\leq k\},
\end{equation}
where $\mathbf{e}_{i}=(0,\cdots,\overset{\left(i\right)}{1},\cdots,0)$ is the standard $i$-th generator of $\mathbb{Z}^{k}$. A  Borel probability measure $\mu$ on $X$ is said to be $\alpha$-\emph{invariant}, if $\mu$ is $f_{i}$-invariant for each $i$. We say a subset $\Gamma\subset X$ is  $\alpha$-\emph{invariant} if $f_{i}(\Gamma)=\Gamma$ for $1\leq i\leq k$.  When $X=M$ is a closed smooth Riemannian manifold $M$, we can consider a $C^{r}, r\ge 1,$ $\mathbb{Z}^{k}$-action $\alpha$ on $M$ from $\mathbb{Z}^{k}$ to Diff$^{r}(M, M)$, where Diff$^{r}(M, M)$ is the space of $C^{r}$ diffeomorphisms equipped with the $C^{r}$-topology.

Subdynamics, especially expansive subdynamics, of a $\mathbb{Z}^k$-action are systematically investigated by Boyle and Lind \cite{Boyle}. Some of the following notations and statements concerning expansiveness are derived from \cite{Boyle}.

For a subset $F\subset \mathbb{R}^k$, put
$$
d_{\alpha}^F (x, y) = \sup\{d(\alpha^{\mathbf{n}}(x),\alpha^{\mathbf{n}}(y)) : \mathbf{n}\in  F \cap \mathbb{Z}^{k}\},\;x,y\in X.
$$
If $F \cap \mathbb{Z}^{k}=\emptyset$, then put $d_{\alpha}^F (x, y)=0$.
Let $|\cdot|$
denote the Euclidean norm on $\mathbb{R}^k$, and for
$\mathbf{v}\in\mathbb{R}^k$, define
$$
\mbox{dist}(\mathbf{v},F) =\inf\{|\mathbf{v}-\mathbf{w}|:\mathbf{w}\in F\}.
$$
For $t>0$ put
$$
F^t=\{\mathbf{v}\in\mathbb{R}^k:\mbox{dist}(\mathbf{v},F)\leq t\},
$$
so that $F^t$ is the result of thickening $F$ by $t$. For a subspace $V$ of $\mathbb{R}^k$, let $\pi_V$ denote orthogonal projection to $V$ along its orthogonal complement $V^{\perp}$, so that $\pi_V+\pi_{V^{\perp}}=Id$. Then clearly
$$
V^t=\{\mathbf{v}\in\mathbb{R}^k :  |\pi_{V^{\perp}}(\mathbf{v})|\le t\}.
$$
For any $\mathbf{v}=(v_1,\cdots,v_k)$ and $r>0$, denote by
$$
B(\mathbf{v}, r)=\{\mathbf{w}:|\mathbf{v}-\mathbf{w}|<r \}
$$
the $r$-ball centered at $\mathbf{v}$.
Let  $\bar{B}(\mathbf{v}, r)$ denote the closure of $B(\mathbf{v}, r)$.

In the remaining of this section, we always assume $\alpha$ is a $\mathbb{Z}^k$-action on $X$, $\Gamma$ is an $\alpha$-invariant subset of $X$ and $F$ is a subset of $\mathbb{R}^k$.

\subsection{Expansiveness}

\begin{definition}

(1) $\alpha$ is called to be \emph{expansive} on $\Gamma$ provided there is an \emph{expansive constant} $\rho_{\Gamma}>0$ such
that for any $x,y\in \Gamma$, $d_{\alpha}^{\mathbb{R}^k}(x, y)\le \rho_{\Gamma}$ implies that $x=y$. In particular, when $\Gamma=X$, we say that $\alpha$ is \emph{expansive}.

(2) We say that $F$ is \emph{expansive} on $\Gamma$ for $\alpha$ if there are \emph{expansive constant} $\rho_{F,\Gamma}>0$ and \emph{thickness} $t=t(F, \rho_{F,\Gamma})>0$ such that for any $x,y\in \Gamma$, $d_{\alpha}^{F^t}(x, y)\le \rho_{F,\Gamma}$ implies that $x=y$.  In particular, when $\Gamma=X$, we say that $F$ is \emph{expansive} for $\alpha$.
\end{definition}

From the above definition, we can see that if $\alpha$ (resp. $F$) is expansive on $\Gamma$ and $\rho_{\Gamma}$ (resp. $\rho_{F,\Gamma}$) is an expansive constant for $\alpha$ (resp. $F$), then any positive number $\rho < \rho_{\Gamma}$ (resp. $\rho < \rho_{F,\Gamma}$) is also an expansive constant for $\alpha$ (resp. $F$). Moreover, for different expansive subsets of $\mathbb{R}^k$ on $\Gamma$ for $\alpha$ usually there are different expansive constants. The following proposition tells us once $\alpha$ is expansive on $\Gamma$ and an expansive constant $\rho_{\Gamma}$ for $\alpha$ is given, via additional thickening, we can use $\rho_{\Gamma}$ as the same expansive constants for all expansive subsets of $\mathbb{R}^k$ for $\alpha$ on $\Gamma$.

\begin{proposition}\label{BP1}

(1) If $F$ is expansive for $\alpha$ on $\Gamma$ and  $E\supset F$, then $E$  is also expansive on $\Gamma$.

(2) If $F$ is expansive for $\alpha$ on $\Gamma$, then every translate $F+\mathbf{v}, \mathbf{v}\in \mathbb{R}^k,$ of $F$  is also expansive for $\alpha$ on $\Gamma$.

(3) Assume $\alpha$ is expansive on $\Gamma$ and $\rho_{\Gamma}$ is an expansive constant of $\alpha$.  Then
 $\rho_{\Gamma}$ is an expansive constant for any expansive subset $F$ of $\mathbb{R}^k$ on $\Gamma$.  In this case,  the corresponding thickness $t=t(F, \rho_{\Gamma})$ is called an \emph{expansive radius} of $F$ on $\Gamma$.

\end{proposition}

\begin{proof}
This proposition is essentially from \cite{Boyle}, and the only difference is that we consider the subsystem of $\alpha$ which restricts to an $\alpha$-invariant set $\Gamma\subset X$. Since some estimates which are applied to this proposition are also helpful for further discussion, we give the outline of the proof.
(1) follows from the observation that if $E\supset F$ then for any $x,y\in X$,
$$
d_{\alpha}^{F^t}(x, y)\le d_{\alpha}^{E^t}(x, y).
$$
(2) follows from (1) and the
observation that for $F\subset \mathbb{R}^k$ and $\mathbf{v}\in \mathbb{R}^k$,
$$
F^t\subset (F+\mathbf{v})^{t+|v|}.
$$
(3) follows from (2) and the fact: when $\alpha$ is expansive on $\Gamma$ with an expansive constant $\rho_{\Gamma}$, then for any $\rho>0$ there exists $r=r(\rho)>0$ such that for any $x,y\in \Gamma$,
$$
d_{\alpha}^{B(\mathbf{0}, r)}(x,y)\leq\rho_{\Gamma} \Longrightarrow d(x,y)\leq\rho.
$$
In fact, for any expansive set $F$ on $\Gamma$ with expansive constant $\rho_{F,\Gamma}$ and thickness $t(F, \rho_{F,\Gamma})$, from the above fact, we can take $\rho_{\Gamma}$ as new expansive constant with $t=t(F, \rho_{F,\Gamma})+r(\rho_{F,\Gamma})+\sqrt{k}$ as the corresponding thickness of $F$ on $\Gamma$.
\end{proof}

\subsection{Shadowing}

 Let $\mathbf{n}=(n_1,\cdots, n_k)\in F\cap \mathbb{Z}^k$. For each $1\le i\le k$, let $\mathbf{n}_{i_+}$ (resp., $\mathbf{n}_{i_-}$) be the element of $F\cap \mathbb{Z}^k$ which is nearest to $\mathbf{n}$ in the $i$-th positive (resp., negative) direction,  and if there is no such  $\mathbf{n}_{i_+}$ (resp., $\mathbf{n}_{i_-}$), then let $\mathbf{n}_{i_+}=\mathbf{n}$ (resp., $\mathbf{n}_{i_-}=\mathbf{n}$).
Hence, the set
$$
\{\mathbf{n}\}\cup\{\mathbf{n}_{i_+}, \mathbf{n}_{i_-}:1\le i\le k\}
$$
 is a ``small" neighborhood of  $\mathbf{n}$ in $F\cap \mathbb{Z}^k$ consisting of  $\mathbf{n}$ and its adjacent elements along the axis directions.

Let $\delta>0$ and $\varepsilon>0$, a set of points $\xi=\{x_{\mathbf{n}} : \mathbf{n}\in  F\cap\mathbb{Z}^k\}$ is called a $\delta$-\emph{pseudo orbit of $F$ for $\alpha$ on $\Gamma$} if $\xi\subset \Gamma$ and
$$
\sup_{\mathbf{n}\in F\cap\mathbb{Z}^k}\max_{1\le i\le k}\max\{d(x_{\mathbf{n}}, \alpha^{\mathbf{n}-\mathbf{n}_{i_-}}(x_{\mathbf{n}_{i_-}})), d(x_{\mathbf{n}_{i_+}}, \alpha^{\mathbf{n}_{i_+}-\mathbf{n}}(x_{\mathbf{n}}))\}\le \delta.
$$
In particular, when $F=\mathbb{R}^k$, the above inequality becomes
$$
\sup_{\mathbf{n}\in \mathbb{Z}^k}\max_{1\le i\le k}d(x_{\mathbf{n}+\mathbf{e}_{i}}, \alpha^{\mathbf{e}_{i}}(x_{\mathbf{n}}))\le \delta.
$$
A point $x\in X$ is called $\varepsilon$-\emph{shadows} the above pseudo orbit $\xi$ if
$$
\sup_{\mathbf{n}\in F\cap\mathbb{Z}^k}d(x_{\mathbf{n}}, \alpha^{\mathbf{n}}(x))\le \varepsilon.
$$

\begin{definition}

(1) We say that $\alpha$ \emph{has the shadowing property on $\Gamma$} provided for any $\varepsilon>0$ there exists $\delta>0$ such that every $\delta$-pseudo orbit for $\alpha$ in $\Gamma$ can be $\varepsilon$-shadowed by some point $x\in X$. In particular, when $\Gamma=X$, we say that $\alpha$ \emph{has the shadowing property}.

(2) We say that $F$ \emph{has the shadowing property for $\alpha$} on $\Gamma$ provided  there exists  $t>0$ satisfying the following property: for any $\varepsilon>0$ there exists $\delta>0$ such that every $\delta$-pseudo orbit of $F^t$ for $\alpha$ in $\Gamma$ can be $\varepsilon$-shadowed by some point $x\in X$. In particular, when $\Gamma=X$, we say that $F$ \emph{has the shadowing property} for $\alpha$.
\end{definition}

\begin{proposition}\label{BP2}
 Assume that $F$ has the shadowing property for $\alpha$ on $\Gamma$.
Then every  translate of $F$ by a vector with integer entries  also has the shadowing property for $\alpha$ on $\Gamma$, i.e.,  for any $\mathbf{m}\in \mathbb{Z}^k$, $F+\mathbf{m}$ has the shadowing property on $\Gamma$.
\end{proposition}

\begin{proof}
Let $F$ have the shadowing property for $\alpha$ on $\Gamma$, then there exists a thickness $t$ of $F$ satisfying the following \emph{property}: for any $\varepsilon>0$ there exists $\delta=\delta(\varepsilon)>0$ such that every $\delta$-pseudo orbit of $F^t$ for $\alpha$ in $\Gamma$ can be $\varepsilon$-shadowed by some point $x\in X$.

Let $\mathbf{m}\in \mathbb{Z}^k$. We will show that $F+\mathbf{m}$ has the shadowing property on $\Gamma$ with the same thickness $t$  as that of $F$. For $\varepsilon>0$, let $\delta>0$ be as above and $\xi=\{x_{\mathbf{n}} : \mathbf{n}\in  (F+\mathbf{m})^t\cap\mathbb{Z}^k\}$ be a $\delta$-pseudo orbit of $(F+\mathbf{m})^t$ in $\Gamma$ for $\alpha$. Let $y_{\mathbf{n}}=x_{\mathbf{n+m}}$ for $\mathbf{n}\in  F^t\cap\mathbb{Z}^k$. It is clearly that $\eta=\{y_{\mathbf{n}} : \mathbf{n}\in  F^t\cap\mathbb{Z}^k\}$ is a $\delta$-pseudo orbit of $F^t$ in $\Gamma$ for $\alpha$. So there exists a point $y\in X$ which $\varepsilon$-shadows $\eta$. Let $x=\alpha^{-\mathbf{n}}(y)$. It is clear that $\xi$ can be $\varepsilon$-shadowed by $x$.
\end{proof}

\subsection{Topologically Anosov}
Combining expansiveness and shadowing together, one has another notion named topologically Anosov.
\begin{definition}
(1) We say that $\alpha$ \emph{is topologically Anosov on $\Gamma$} provided $\alpha$ is expansive and has the shadowing property on $\Gamma$. In particular, when $\Gamma=X$, we say that $\alpha$ \emph{is topologically Anosov}.

(2) We say that $F$ \emph{is topologically Anosov on $\Gamma$ for $\alpha$} provided $F$ is expansive and has the shadowing property for $\alpha$ on $\Gamma$. In particular, when $\Gamma=X$, we say that $F$ \emph{is topologically Anosov} for $\alpha$.
\end{definition}

From now on we shall only be concerned with subsets that are subspaces of $\mathbb{R}^k$.
In order to discuss sets of subspaces, we recall the Grassmann manifold $\mathbb{G}_j=\mathbb{G}_{j,k}$
of all $j$-dimensional subspaces (or $j$-planes) of $\mathbb{R}^k$. The topology of $\mathbb{G}_j$ is induced
by the metric for which the distance between two subspaces is the Hausdorff metric
distance between their intersections with the unit sphere in $\mathbb{R}^k$. Then $\mathbb{G}_j$ is a
compact manifold of dimension $j(k-j)$.

\begin{definition}
For each $1\le j\le k$, define
\begin{align*}
\mathbb{E}_j(\alpha, \Gamma)&=\{V\in \mathbb{G}_j : V \text{ is expansive on } \Gamma \text{ for }\alpha\};\\
\mathbb{S}_j(\alpha, \Gamma)&=\{V\in \mathbb{G}_j: V \text{ has the shadowing property  on } \Gamma \text{ for }\alpha\};\\
\mathbb{A}_j(\alpha, \Gamma)&=\{V\in \mathbb{G}_j: V \text{ is topologically Anosov  on } \Gamma \text{ for }\alpha\}.
\end{align*}
\end{definition}

Clearly, we have that
$\mathbb{A}_j(\alpha, \Gamma)=\mathbb{E}_j(\alpha, \Gamma)\cap \mathbb{S}_j(\alpha, \Gamma).$
From \cite{Boyle}, $\mathbb{E}_j(\alpha, \Gamma)$ is always open in $\mathbb{G}_j$.

\begin{proposition}\label{BP3}

 (1) If $V\in \mathbb{E}_j(\alpha, \Gamma)$ and $W$ is a subspace of $\mathbb{R}^k$ containing $V$, then $W\in \mathbb{E}_{\dim W}(\alpha, \Gamma)$. Hence, once $\mathbb{E}_j(\alpha, \Gamma)\neq \emptyset$ for some $j$ then $\mathbb{E}_i(\alpha, \Gamma)\neq\emptyset,\;j\le i\le k$.

(2) If $V\in \mathbb{A}_j(\alpha, \Gamma)$,  $W$ is a subspace of $\mathbb{R}^k$ containing $V$, then $W\in \mathbb{A}_{\dim W}(\alpha, \Gamma)$. Hence, once $\mathbb{A}_j(\alpha, \Gamma)\neq \emptyset$ for some $j$ then $\mathbb{A}_i(\alpha, \Gamma)\neq\emptyset,\;j\le i\le k$.

\end{proposition}

\begin{proof}
(1) follows from Proposition \ref{BP1}. We only need to prove (2).

Let $V\in \mathbb{A}_j(\alpha, \Gamma)$ and $W$ be a subspace of $\mathbb{R}^k$ containing $V$. Note that $\mathbb{A}_j(\alpha, \Gamma)=\mathbb{E}_j(\alpha, \Gamma)\cap \mathbb{S}_j(\alpha, \Gamma)$. By $V\in \mathbb{E}_j(\alpha, \Gamma)$, there are constants $\rho>0$ and thickness $t_1>0$ such that for any $x,y\in \Gamma$,
\begin{equation}\label{expan}
d_{\alpha}^{V^{t_1}}(x, y)\le \rho \Longrightarrow x=y.
\end{equation}
By $V\in \mathbb{S}_j(\alpha, \Gamma)$, there exists  $t_2>0$ satisfying the following property: for any $\varepsilon_1>0$ there exists $\delta_1>0$ such that every $\delta_1$-pseudo orbit of $V^{t_2}$ for $\alpha$ in $\Gamma$ can be $\varepsilon_1$-shadowed by some point $x\in X$. Generally, we do not know $t_1$ and $t_2$ which is bigger. However, from the fact we have used to get (3) of Proposition \ref{BP1}, we can increase $t_1$ and meanwhile decrease $\rho$ properly such that (\ref{expan}) holds. Therefore, we can assume $t_1>t_2+\sqrt{k}$.

We firstly show that $V$ is expansive and has the shadowing property for $\alpha$ on $\Gamma$ with the same thickness $t_1$.  Since $t_1>t_2+\sqrt{k}$, we can take a neighborhood $\mathbf{U}(\mathbf{0})$ of $\mathbf{0}$ in $\mathbb{R}^k$ such that
$$
V^{t_1}=\bigcup_{\mathbf{n}\in \mathbf{U}(\mathbf{0})\cap\mathbb{Z}^k}(V+\mathbf{n})^{t_2}.
$$
For each $\mathbf{n}\in \mathbf{U}(\mathbf{0})\cap\mathbb{Z}^k$, denote $V_{\mathbf{n}}=V+\mathbf{n}$. By Proposition \ref{BP2}, each above $V_{\mathbf{n}}$ has the shadowing property with the thickness $t_2$.
Since the generators $f_i, 1\le j\le k$, of $\alpha$ are equi-continuous, for any $0<\varepsilon<\frac{\rho}{4}$ we can take positive numbers $\varepsilon_1=\varepsilon_1(\varepsilon), \delta_1=\delta_1(\varepsilon_1)$ and $\delta$ satisfying the following properties:
for any $\delta$-pseudo orbit $\xi=\{x_{\mathbf{n}} : \mathbf{n}\in  V^{t_1}\cap\mathbb{Z}^k\}$ in $\Gamma$,  its subset $\xi_{\mathbf{n}}=\{x_{\mathbf{m}} : \mathbf{m}\in  V_{\mathbf{n}}^{t_2}\cap\mathbb{Z}^k\}$, for each $\mathbf{n}\in \mathbf{U}(\mathbf{0})\cap\mathbb{Z}^k$,  is a $\delta_1$-pseudo orbit of $V_{\mathbf{n}}^{t_2}$, hence $\xi_{\mathbf{n}}$ can be $\varepsilon_1$-shadowed by some point $z_{\mathbf{n}}\in X$, and therefore $\xi$ can be $\varepsilon$-shadowed by  $\alpha^{-\mathbf{n}}(z_{\mathbf{n}})$. By the choices of $\varepsilon$, we get that $\alpha^{-\mathbf{n}}(z_{\mathbf{n}})=z_{\mathbf{0}}$ for each $\mathbf{n}\in \mathbf{U}(\mathbf{0})\cap\mathbb{Z}^k$. Thus, we conclude that $\xi$ can be $\varepsilon$-shadowed by the point $z_{\mathbf{0}}$. Now, we get that $V$ is expansive and has the shadowing property for $\alpha$ on $\Gamma$ with the same thickness $t_1$. In this case we also say that $\alpha$ \emph{is topologically Anosov on} $V^{t_1}$ on $\Gamma$.

Using a similar method to move  $V^{t_1}$ along $W$ step by step, we can get larger and larger sets on which $\alpha$ is topologically Anosov on $\Gamma$. Eventually, we get that $\alpha$ is topologically Anosov on $W^{t_1}$  on $\Gamma$. Hence, $W\in \mathbb{E}_{\dim W}(\alpha, \Gamma)$.

This completes the proof of (2).
\end{proof}

\begin{remark}\label{BP4}
 Using the similar idea in proving (2) of Proposition \ref{BP3}, we have the following property:
if $F$ is topologically Anosov for $\alpha$ on $\Gamma$  then every translate $F+\mathbf{v}, \mathbf{v}\in \mathbb{R}^k,$ of $F$  is also topologically Anosov for $\alpha$ on $\Gamma$. It is a generalization of the result in Proposition \ref{BP2} under a stronger condition.
\end{remark}

\section{\texorpdfstring{Shadowing for certain classical $\mathbb{Z}^{k}$-actions}{Shadowing for certain classical Z\^{}k-actions}}

In this section, we consider some classical examples for $\mathbb{Z}^{k}$-actions. One type is symbolic system, the other type is the ``linear" system on torus.

\subsection{Symbolic systems}

Let $\mathcal{A}$ be a finite alphabet, equipped with the discrete topology. Define a metric $\varrho$ on $\mathcal{A}$ by $\varrho(a,b)=1$ if $a\neq b$ and $0$ otherwise.
Put $X=\mathcal{A}^{\mathbb{Z}^k}$, equipped with the
product topology. Thus a point $x\in X$ has the form
$(x(\mathbf{i}))_{\mathbf{i}\in \mathbb{Z}^k}$.
Define a metric $d$ on $X$  by
\begin{equation}\label{metric}
d(x, y)=\sum_{\mathbf{i}\in \mathbb{Z}^k}\frac{\varrho(x(\mathbf{i}),y(\mathbf{i}))}{2^{|\mathbf{i}|}}
\end{equation}
for $x=(x(\mathbf{i}))_{\mathbf{i}\in \mathbb{Z}^k}, y=(y(\mathbf{i}))_{\mathbf{i}\in \mathbb{Z}^k}$.
Define the \emph{shift action} $\alpha$ of $\mathbb{Z}^k$ on $X$ by
$$
(\alpha^{\mathbf{n}}x)(\mathbf{i}) = x(\mathbf{n}+\mathbf{i}).
$$

From \cite{Boyle}, $\alpha$ is expansive and $\mathbb{E}_j(\alpha)=\emptyset$ for $0\le j \le k-1$ (i.e., there is no proper subspace $V\subset \mathbb{R}^{k}$ can be expansive for $\alpha$). So, $\mathbb{A}_j(\alpha)=\emptyset$ for $0\le j \le k-1$. Now we consider the shadowing property for subspaces of $\mathbb{R}^{k}$.

\begin{proposition}\label{prop1}
Let $\alpha$ be the $\mathbb{Z}^k$ shift action on $X=\mathcal{A}^{\mathbb{Z}^k}$. Then $\alpha$ has the shadowing property and hence is topologically Anosov. Moreover, $\mathbb{S}_j(\alpha)=\mathbb{G}_j$ for $0\le j \le k-1$.
\end{proposition}
\begin{proof}
From the definition of the metric $d$ on $X=\mathcal{A}^{\mathbb{Z}^k}$, we can see that two points are close in this metric provided their coordinates agree in a large neighborhood of the origin, and in particular, $d(x, y)< 1$ implies that $x(\mathbf{0})= y(\mathbf{0})$. Precisely, for any $\varepsilon>0$ there exists $r>0$ such that for $x=(x(\mathbf{i}))_{\mathbf{i}\in \mathbb{Z}^k}, y=(y(\mathbf{i}))_{\mathbf{i}\in \mathbb{Z}^k}\in X$,
$$
x(\mathbf{i})=y(\mathbf{i}) \text{ for }\mathbf{i}\in B(\mathbf{0}, r)\Longrightarrow d(x,y)\le \varepsilon.
$$
For the above $r$, take $\delta>0$
$$
d(x,y)\le \delta\Longrightarrow x(\mathbf{i})=y(\mathbf{i}) \text{ for }\mathbf{i}\in B(\mathbf{0}, r+2\sqrt{k}).
$$
By a standard discussion we get that for any $\delta$-pseudo orbit $\xi=\{x_{\mathbf{n}} : \mathbf{n}\in \mathbb{Z}^k\}$ of $\alpha$, there exists a point $x^*=(x^*(\mathbf{i}))_{\mathbf{i}\in \mathbb{Z}^k}$, which is defined by $x^*(\mathbf{i})=x_{\mathbf{i}}({\mathbf{0}})$,  $\varepsilon$-shadows it. Therefore, $\alpha$ has the shadowing property. Moreover, since $\alpha$ is expansive (assume $\rho$ is an expansive constant), once $\varepsilon<\frac{\rho}{2}$ we conclude that the shadowing point $x^*$ is unique.

By a similar discussion, we get $\mathbb{S}_j(\alpha)=\mathbb{G}_j, 0\le j \le k-1$ immediately.
\end{proof}

Now we consider the example due to Ledrappier(\cite{Ledrappier}).

\begin{example}\label{example1}
Let $\mathcal{A}=\mathbb{Z}/2\mathbb{Z}$ and $\alpha$ be the $\mathbb{Z}^2$ shift action on $X=\mathcal{A}^{\mathbb{Z}^2}$. Consider the compact $\alpha$-invariant subset $\Gamma$ of $X$
defined by the condition
\begin{equation}\label{Ledrappier}
 x((i, j)) + x((i+1, j)) + x((i, j+1))=0\;\;(\emph{\mbox{mod} }2)
\end{equation}
for all $i, j \in \mathbb{Z}$. Let $L_{\theta}$ denote the line making angle $\theta$ with
the positive horizontal axis, then $\mathbb{G}_1=\{L_{\theta} : 0\le \theta <\pi\}$. From \cite{Boyle}, $\alpha$ is expansive on $\Gamma$ and
$$
\mathbb{E}_1(\alpha, \Gamma)=\mathbb{G}_1\setminus\{L_{0},L_{\frac{\pi}{4}},L_{\frac{3\pi}{4}}\}.
$$
By a similar discussion as in Proposition \ref{prop1}, we have that $\alpha$ has the shadowing property on $\Gamma$, and $\mathbb{S}_1(\alpha, \Gamma)=\mathbb{G}_1$.
\end{example}

\subsection{\texorpdfstring{Linear systems on $\mathbb{R}^m$ and linear reduced systems on $\mathbb{T}^m$}{Linear systems on R\^{}m and linear reduced systems on T\^{}m}}

In \cite{Pilyugin1}, the shadowing property for linear $\mathbb{Z}^k$-actions on $\mathbb{C}^m$ is considered. Let $A_i, 1\le i\le k$ be pairwise commuting upper triangular matrices with complex entries and $\alpha$ be the induced $\mathbb{Z}^k$-action on $\mathbb{C}^m$.
Denote by $\lambda_{i,j}$ the $j$th diagonal element (i.e., the $(j, j)$ entry) of $A_i$. Applying the matrix analysis technique, Pilyugin and  Tikhomirov \cite{Pilyugin1} gave a characterization of the linear $\mathbb{Z}^k$-actions on $\mathbb{C}^m$ which has the shadowing property.

\begin{proposition} [Theorem 2 of \cite{Pilyugin1}]
The following statements are equivalent for the above $\alpha$:

(1) $\alpha$ has the Lipschitz shadowing property in the following sense: there exists
a constant $L>0$ such that for any $\delta$-pseudo orbit $\xi=\{x_{\mathbf{n}} : \mathbf{n}\in  \mathbb{Z}^k\}$
there is a point $x$ satisfying
$$
|\alpha^\mathbf{n}(x)-x_{\mathbf{n}}|\le L\delta,\;\; \mathbf{n}\in\mathbb{Z}^k;
$$

(2) for any $j\in \{1,\cdots,m\}$ there exists $i\in \{1,\cdots,k\}$ such that $|\lambda_{i,j}|\neq1$;

(3) there is no vector $v\neq0$ such that
$$
A_iv=\mu_iv, \;\;i=1,\cdots,k,\;\; \text{ where } |\mu_i|=1.
$$
\end{proposition}

The above Proposition and Theorem 1 of \cite{Pilyugin1} tell us that a linear $\mathbb{Z}^k$-action on $\mathbb{C}^m$ has the Lipschitz shadowing property if and only if there exists at least one hyperbolic 1-dimensional \emph{rational}, subspaces of $\mathbb{R}^k$, here we say the a 1-dimensional subspace $V$ of $\mathbb{R}^k$ is rational we mean that $V$ passes through some integer lattices except for $\mathbf{0}$. When each of the above generators $A_i, 1\le i\le k,$ has integer entries and whose determinant is equal to $\pm 1$, we can get, from the above theorem, a characterization of the induced $\mathbb{Z}^k$-action $\alpha$ on the torus $\mathbb{T}^m$ which has the shadowing property.

 A natural question is as follows.
\begin{question}
Can we give a condition under which a smooth $\mathbb{Z}^k$-action $\alpha$ on a closed Riemannian manifold has the shadowing property?  Furthermore, can we get a characterization for the subsystems of $\alpha$ which has the shadowing property?
\end{question}

The main task of this paper is to answer the above question, and the next section is devoted to this topic. From the main results (Theroem A and Theorem B) of next section, we have more information concerning shadowing properties for the following particular examples.

\begin{example}\label{example3}
Let $\alpha$ be the ${\mathbb{Z}}^2$-action  on the torus
$\mathbb{T}^2$ with the generators $\{f_1,f_2\}$ which are induced by the hyperbolic automorphisms
$$
A_1=\left(
         \begin{array}{rr}
              2 &1\\
              1 &1
          \end{array} \right)\;\; \mbox{and}\;\;
          A_2=A_1^{-1}=\left(
         \begin{array}{rr}
              1 &-1\\
              -1 &2
          \end{array} \right),
$$
 respectively. The eigenvalues
of $A_1$ are $\lambda_1=\frac{3+\sqrt{5}}{2}$ and $\lambda_2=\frac{3-\sqrt{5}}{2}$, and let $E_1$ and $E_2$ be the corresponding
eigenspaces in ${\mathbb{R}}^2$. Since $A_2$ is the inverse of $A_1$, $A_2$ has eigenvalues $\mu_1=\lambda_2$ and $\mu_2=\lambda_1$, with
eigenspaces $F_1=E_2$ and $F_2=E_1$. We can see that
$$
\mathbb{E}_1(\alpha)=\mathbb{S}_1(\alpha)=\mathbb{A}_1(\alpha)=\mathbb{G}_1\setminus \{L_1\},
$$
 where $L_1$ is the line in $\mathbb{G}_1$ with slope 1. $\alpha$ does not have the shadowing property along $L_1$ since $\alpha^{(1,1)}=Id$. Since $\mathbb{A}_1(\alpha)\neq \emptyset$, by Proposition \ref{BP3}, $\alpha$ is topologically Anosov.
\end{example}

\begin{example}\label{example3.5}
Let $\alpha$ be the ${\mathbb{Z}}^2$-action  on the torus
$\mathbb{T}^3$ with the generators $\{f_1,f_2\}$ which are induced by the partially hyperbolic automorphisms
$$
A_1=\left(
         \begin{array}{rrr}
               1 &0 &0\\
              0 &2 &1\\
              0 &1 &1
          \end{array} \right)\;\; \mbox{and}\;\;
          A_2=A_1^{-1}=\left(
         \begin{array}{rrr}
              1 &0 &0\\
              0 &1 &-1\\
              0 &-1 &2
          \end{array} \right),
$$
 respectively. We can see that $\mathbb{E}_1(\alpha)=\mathbb{S}_1(\alpha)=\mathbb{A}_1(\alpha)=\emptyset$, and $\alpha$ does not have the shadowing property.
 Let $L_1$ be the line in $\mathbb{G}_1$ with slope 1, then for any $L\in \mathbb{G}_1\setminus \{L_1\}$, $\alpha$ has the quasi-shadowing property (see Definition \ref{QS}) along $L$.
\end{example}

\begin{example}\label{example4}[Example 2.10 of \cite{Boyle}]
Let
$$
A_1=\left(
         \begin{array}{rr}
              1 &1\\
              2 &1
          \end{array} \right)\;\; \mbox{and}\;\;
          A_2=\left(
         \begin{array}{rr}
              2 &1\\
              3 &2
          \end{array} \right).
$$
Clearly, they induce automorphisms of $\mathbb{T}^2$ (which do not commute). The eigenvalues
of $A_1$ are $\lambda_1=1+\sqrt{2}$ and $\lambda_2=1-\sqrt{2}$, and let $E_1$ and $E_2$ be the corresponding
eigenspaces in ${\mathbb{R}}^2$. Similarly, $A_2$ has eigenvalues $\mu_1=2+\sqrt{3}$ and $\mu_2=2-\sqrt{3}$, with
eigenspaces $F_1$ and $F_2$.

Let
$$
B_1=A_1\otimes I=\left(
         \begin{array}{rrrr}
              1&1&0&0\\
              2&1&0&0\\
              0&0& 1&1\\
              0&0&2&1
          \end{array} \right)\;\; \mbox{and}\;\;
B_2=I\otimes A_2=\left(
         \begin{array}{rrrr}
              2&0&1&0\\
              0&2&0&1\\
              3&0& 2&0\\
              0&3&0&2
          \end{array} \right).
$$
Both $B_1$ and $B_2$ act on
$$
{\mathbb{R}}^2\otimes {\mathbb{R}}^2=\bigoplus_{i,j=1}^2E_i\otimes F_j
$$
and they commute. Hence
they define a ${\mathbb{Z}}^2$-action on $\mathbb{T}^4$  given by $\alpha^{(n_1,n_2)} = B_1^{n_1}B_2^{n_2} = A_1^{n_1}\otimes A_2^{n_2}$. The
1-dimensional spaces $E_i\otimes F_j$ are common eigenspaces for $B_1$ and $B_2$, and $\alpha^{(n_1,n_2)}$  on
$E_i\otimes F_j$ is multiplication by $\lambda_i^{n_1}\mu_j^{n_2}$.

For $i=1, 2$, let $L_i$ be the line in ${\mathbb{R}}^2$ with slope $-\log |\lambda_i|= \log |\mu_1|$. We can see that
$$
\mathbb{E}_1(\alpha)=\mathbb{S}_1(\alpha)=\mathbb{A}_1(\alpha)=\mathbb{G}_1\setminus \{L_1, L_2\}.
$$
Similar to that in Example \ref{example3}, $\alpha$ does not have the shadowing property along $L_1$ and $L_2$. Since $\mathbb{A}_1(\alpha)\neq \emptyset$, by Proposition \ref{BP3}, $\alpha$ is topologically Anosov.
\end{example}

\section{\texorpdfstring{Shadowing in smooth $\mathbb{Z}^{k}$-actions}{Shadowing in smooth Z\^{}k-actions}}

In this section, we investigate the shadowing property for smooth $\mathbb{Z}^k$-actions on and give the characterization of the subspace $V\subset\mathbb{R}^k$ along which $\alpha$ has the shadowing property.

We first recall some fundamental properties of  $C^r,r\ge 1,$ $\mathbb{Z}^k$-actions. Let $M$ be an $m$-dimensional closed Riemannian manifold.
Let $\alpha:\mathbb{Z}^k\longrightarrow C^r(M, M)$ be a $C^r$ $\mathbb{Z}^k$-action on $M$  whose generators, $f_i, 1\le i\le k$, are pairwise commuting $C^{r}$ diffeomorphisms.
A Borel probability measure $\mu$ on $M$ is called  \emph{$\alpha$-invariant} (resp. \emph{ergodic}) if
 it is invariant (resp. ergodic) with respect to each $f_i,1\le i\le k$.

Let $\mu$ be an $\alpha$-ergodic measure,  $\Gamma$ the Oseledec set of $\alpha$ and
$$
\{(\lambda_{i,j},m_j): 1\le i\le k, 1\le j\le s\}
$$
the pectrum of $\alpha$ (see \cite{Hu}, \cite{Kalinin} and \cite{Brown} for example).
The following basic notions come from Kalinin, Katok and Hertz's paper \cite{Kalinin1}. $T_{\Gamma}M=\bigoplus_{j=1}^{s} E_j$ is called the \emph{Lyapunov decomposition} for $\alpha$. By For each Lyapunov distribution $E_j, 1\le j\le s$, one can define a linear functional on $\mathbb{Z}^k$ and then extend it to a linear functional on $\mathbb{R}^k$ as follows,
$$
\chi_j: \mathbb{R}^k \longrightarrow \mathbb{R},\;\;\mathbf{v}=(v_1,\cdots,v_k)\longmapsto\sum_{i=1}^kv_i\lambda_{i,j}.
$$
 The hyperplanes ker$\chi_j\subset\mathbb{R}^k, 1\le j\le s$, are called
the \emph{Lyapunov hyperplanes} and the connected components of $\mathbb{R}^k\setminus\cup_{j=1}^s$ ker$\chi_j$ are
called the \emph{Weyl chambers} of $\alpha$. Clearly, each Weyl chamber is an open convex cone in $\mathbb{R}^k$. The elements (or say, vectors) in the union of the Lyapunov
hyperplanes are called \emph{singular}, and the elements in the union of the Weyl
chambers are called \emph{regular}. So, a singular  vector $\mathbf{v}=(v_1,\cdots,v_k)\in\mathbb{R}^k$ satisfies  the following condition
\begin{equation}\label{Singular}
 \sum_{i=1}^kv_i\lambda_{i,j}= 0\;\;\text{ for at least one }1\le j\le s;
\end{equation}
while, a regular vector $\mathbf{v}=(v_1,\cdots,v_k)\in\mathbb{R}^k$ satisfies  the following condition
\begin{equation}\label{Regular}
 \sum_{i=1}^kv_i\lambda_{i,j}\neq 0\;\;\text{ for any }1\le j\le s.
\end{equation}


For a nonzero vector $\mathbf{v}=(v_1,\cdots,v_k)\in\mathbb{R}^k$, let $L_\mathbf{v}$ be the line in which the vector $\mathbf{v}$ lies.
We call $L_\mathbf{v}$ (or $\mathbf{v}$)  is \emph{rational} if $L_\mathbf{v}\cap \mathbb{Z}^k\setminus \{\mathbf{0}\}\neq \emptyset$, i.e., the line $L_\mathbf{v}$ passes through some integer lattices except for $\mathbf{0}$. Otherwise, we say $L_\mathbf{v}$ (or $\mathbf{v}$) is \emph{irrational}.

Lemma 1 of \cite{Pilyugin1} gives us a tool to justify  when there exist regular vectors and hence the union of the Weyl chambers is not empty. Precisely, Lemma 1 of \cite{Pilyugin1} shows that if for each $1\le j\le s$, there exists $1\le i\le k$ such that $\lambda_{i,j}\neq 0$, then there exists an integer (and hence, rational) vector $\mathbf{v}=(v_1,\cdots,v_k)\in\mathbb{Z}^k\setminus \{\mathbf{0}\}$ satisfying (\ref{Regular}). In fact, since each Weyl chamber is an open convex cone in $\mathbb{R}^k$, once there is a regular vector, there is a neighbourhood of it in which all the vectors are regular.

\subsection{\texorpdfstring{Shadowing for $\alpha$ along subspace containing regular vector(Proof of Theorem A)}{Shadowing for α along subspace containing regular vector(Proof of Theorem A)}}

The strategy to prove Theorem A is as follows: firstly choose a ``thin tube" around the 1-dimensional subspace $V$ of $\mathbb{R}^k$ which contains $\mathbf{v}$, establish a Lipschitz shadowing result for a carefully selected subsystem in this tube, then deduce a shadowing property for the restriction of $\alpha$ to this tube, and hence by Proposition \ref{BP3} we get the shadowing property for any higher dimensional subspace of $\mathbb{R}^k$ containing $\mathbf{v}$ inductively.

We will show Theorem A in two cases: (1) $\mathbf{v}$ is rational; (2) $\mathbf{v}$ is irrational. We first consider the easier case.

\begin{proposition}\label{case1}
Let $\alpha$ be a $C^r$ $\mathbb{Z}^k$-action on $M$ as in Theorem A. If $\mathbf{v}=(v_1,\cdots,v_k)\in\mathbb{R}^k$ is a rational regular vector, then $L_\mathbf{v}\in\mathbb{A}_1(\alpha, \Gamma)$.
\end{proposition}

Let us describe the idea to prove Proposition \ref{case1}.  For a rational regular vector $\mathbf{v}$, we will take a suitable non-zero $\mathbf{n}=(n_1,\cdots,n_k)\in\mathbb{Z}^k\cap L_\mathbf{v}$ such that the diffeomorphism $\alpha^{\mathbf{n}}$ is hyperbolic on $\Gamma$. By the classical properties of hyperbolic set of a diffeomorphism, we have that $\alpha^{\mathbf{n}}$ has the Lipschitz shadowing property and is expansive on $\Gamma$, and then by choosing an appropriate thickness of $L_\mathbf{v}$ we deduce $L_\mathbf{v}\in\mathbb{A}_1(\alpha, \Gamma)$. So we first state some basic concepts and properties of hyperbolic set of a diffeomorphism.

\begin{definition}
Let $g:M\longrightarrow M$ be a diffeomorphism and $\Lambda\subset M$ (where $\Lambda$ is not necessarily compact) be $g$-invariant.  We call that $\Lambda$ is a \emph{hyperbolic set} of $g$ if there exist an invariant splitting $T_{\Lambda}M=E^{s}\bigoplus E^{u}$ and a constant $0<\lambda<1$ such that for $m\ge 0$,
$$
\|Dg^m(x)v\|\leq \lambda^m\|v\|\;\; \mbox{for}\;\; v\in E^s(x)
$$
and
$$
\|Dg^m(x)v\|\geq \lambda^{-m}\|v\| \;\; \mbox{for}\;\; v\in E^u(x).
$$
We call the number $\lambda$ a \emph{hyperbolicity constant} of $g$
on $\Lambda$.
\end{definition}
 By Theorem 4.3 of the monograph of Wen \cite{Wen}, for the hyperbolic set $\Lambda$ of $g$, the corresponding hyperbolic decomposition $T_{\Lambda}M=E^s\bigoplus E^u$ is continuous, and moreover, it can be continuously extended to the closure  of $\Lambda$ such that $g_{\overline{\Lambda}}$, which is the restriction of $g$ on $\overline{\Lambda}$, is also hyperbolic. By this observation, we have that there exists a positive number $r$
 \begin{equation}\label{angle}
 \sup_{x\in \Lambda} \angle(E^s(x),E^u(x))\ge r.
 \end{equation}

Let $\varepsilon,\delta>0$. A sequence of points $\xi=\{x_{p}\}_{p\in\mathbb{Z}}$ is said to be a
$\delta$\emph{-pseudo orbit} for $g$ on $\Lambda$ if $\xi\subset \Lambda$ and $\sup_{p\in\mathbb{Z}}d(g(x_p),x_{p+1})\leq \delta.$
We say that $x\in M$ $\varepsilon$-\emph{shadows}  $\xi$ if $\sup_{p\in\mathbb{Z}}d(g^p(x),x_p)\leq \varepsilon.$

 We say that $g$ \emph{has the Lipschitz shadowing property on $\Lambda$} provided there exist $\delta_0, L>0$ such that any  $\delta$-pseudo orbit for $g$ in $\Lambda$ with $\delta\le \delta_0$ can be $L\delta$-shadowed by some point $x\in M$.
 We say that $g$ is \emph{expansive} on $\Lambda$ provided there is an \emph{expansive constant} $\rho_{\Lambda}>0$ such
that for any $x,y\in \Lambda$, $\sup_{p\in\mathbb{Z}}d(g^p(x), g^p(y))\le \rho_{\Lambda}$ implies that $x=y$.
We say that $g$ \emph{is topologically Anosov on $\Lambda$} provided $g$ is expansive and has the shadowing property on $\Lambda$.

The following properties about  hyperbolic set of a diffeomorphism are classical, we can see them in many books on differentiable dynamical systems (see, \cite{Katok}, \cite{Pilyugin} and \cite{Wen}, for example).
\begin{lemma}\label{Wen}
Let $g:M\longrightarrow M$ be a diffeomorphism and $\Lambda\subset M$ a hyperbolic set of $g$. Then  $g$ has the Lipschitz shadowing property and is expansive,  and hence is topologically Anosov on $\Lambda$.
\end{lemma}

\begin{proof}[Proof of Proposition \ref{case1}]
Let $\mathbf{v}=(v_1,\cdots,v_k)$ be a rational nonzero vector which satisfies (\ref{Regular}). If there exists some components in $\mathbf{v}$ are zero, say $v_1, \cdots, v_l=0$ for some $1\le l<k$, we may first consider the $\mathbb{Z}^{k-l}$-action $\alpha'$ on $M$ generated by $\{f_{l+1},\cdots,f_k\}$ and show that $L_\mathbf{v'}\in\mathbb{A}_1(\alpha', \Gamma)$ for $\mathbf{v}'=(v_{l+1},\cdots,v_k)$, and hence get $L_\mathbf{v}\in\mathbb{A}_1(\alpha, \Gamma)$.
So, we assume that each component $v_i$ in $\mathbf{v}$ is not $0$ in the following.

By (\ref{Regular}), we can write the set $\{1,\cdots,s\}$ into a disjoint union of the following two subsets
\begin{equation}\label{J1J2}
J_{1}=\{j:\sum_{i=1}^kv_i\lambda_{i,j}<0\}\;\;\text{ and } J_{2}=\{j:\sum_{i=1}^kv_i\lambda_{i,j}>0\}.
\end{equation}
Take $a>0$ small enough such that
\begin{equation}\label{bv12}
b_{\mathbf{v},1}:=\max_{j\in J_1}\sum_{i=1}^kv_i(\lambda_{i,j}+ a)<0\text{ and }
b_{\mathbf{v},2}:=\min_{j\in J_2}\sum_{i=1}^kv_i(\lambda_{i,j}-a)>0.
\end{equation}
Observe that, $b_{l\mathbf{v},1}$ decreases while $b_{l\mathbf{v},2}$ increases,  as $l>0$ increases.
Let
\begin{equation}\label{bv}
b_{\mathbf{v}}=\min\{-b_{\mathbf{v},1}, b_{\mathbf{v},2}\}.
\end{equation}
By the basic assumption, the convergence in (\ref{ULY}) is uniform,
hence for the above $a>0$ there exist $N>0$ such that when $|n|\ge N$, the inequalities in (\ref{Uniformity}) hold.
Since we have assumed that each component $v_i$ in $\mathbf{v}$ is not $0$, we can take a nonzero vector $\mathbf{n}^*=(n^*_1,\cdots,n^*_k)\in L_\mathbf{v}\cap\mathbb{Z}^k$ such that $|n^*_i|>N$ for each $i=1,\cdots,k$.

We \emph{claim} that for the above $\mathbf{n}^*$, the diffeomorphism $g:=\alpha^{\mathbf{n}^*}$ on $M$ is hyperbolic on $\Gamma$. For each $x\in \Gamma$, denote
\begin{equation}\label{EsEu}
E^s(x)=\bigoplus_{j\in J_{1}}E_j(x)\;\text{ and }\;E^u(x)=\bigoplus_{j\in J_{2}}E_j(x).
\end{equation}
From the choice of $\mathbf{n}^*$ we conclude that for each $0\neq v\in E^s(x)$ and $m>0$,
$$
\|Dg^m(x)v\|=\|D(\alpha^{\mathbf{n}^*})^m(x)v\|=\|D(f_1^{n^*_1}\circ\cdots\circ Df_k^{n^*_k})^m(x)v\|\le (e^{-b_{\mathbf{n}^*}})^m\|v\|;
$$
and for each $0\neq u\in E^u(x)$ and $m>0$,
$$
\|Dg^m(x)u\|=\|D(\alpha^{\mathbf{n}^*})^m(x)u\|=\|D(f_1^{n^*_1}\circ\cdots\circ Df_k^{n^*_k})^m(x)v\|\ge (e^{b_{\mathbf{n}^*}})^m\|u\|.
$$
Note that $b_{\mathbf{n}^*}>0$, hence $e^{-b_{\mathbf{n}^*}}<1$ and $e^{b_{\mathbf{n}^*}}>1$.  Thus the claim holds.

By Lemma \ref{Wen}, we have that the above $\alpha^{\mathbf{n}^*}$ has the Lipschitz shadowing property (with the corresponding constants $\delta^*_0, L^*>0$) and is expansive (with the expansiveness constant $\rho^*>0$) on $\Gamma$. Let $\mathbf{n}^\diamond=(n^\diamond_1,\cdots,n^\diamond_k)$ be one of the two elements in  $L_\mathbf{v}\cap \mathbb{Z}^k\setminus \{\mathbf{0}\}$ with the smallest norm. Let $l^*=\frac{|\mathbf{n}^\diamond|}{|\mathbf{n}^*|}$, then $l^*$ is an integer and $\mathbf{n}^*=l^*\mathbf{n}^\diamond$.
Define
$$
L_1=\max_{x\in M}\|D\alpha^{\mathbf{n}^\diamond}(x)\|, L_2=\max\{1,\frac{L^{l^*}_1-1}{L_1-1}\} \text{ and }\delta_0=\frac{\delta_0^*}{L_2}.
$$
Let $\xi=\{x_{p}\}_{p\in\mathbb{Z}}$ be a $\delta$-pseudo orbit for $\alpha^{\mathbf{n}^\diamond}$ in $\Gamma$ with $\delta\le \delta_0$, then its subsequence $\xi'=\{x_{l^*p}\}_{p\in\mathbb{Z}}$ must be a $L_2\delta(<\delta_0^*)$-pseudo orbit for $\alpha^{\mathbf{n}^*}$, and hence can be $L^*L_2\delta$-shadowed by some point $x\in M$ for $\alpha^{\mathbf{n}^*}$. Therefore, $\xi$  can be $(1+L^*L_1^{l^*})L_2\delta$-shadowed by $x$ with respect to $\alpha^{\mathbf{n}^\diamond}$. Denote $L=(1+L^*L_1^{l^*})L_2$. Thus $\alpha^{\mathbf{n}^\diamond}$ has the Lipschitz shadowing property with the corresponding constants $\delta_0, L>0$. Furthermore, it is not difficult to see that $\alpha^{\mathbf{n}^\diamond}$ is expansive on $\Gamma$ with the expansiveness constant $\rho=\frac{\rho^*}{\max\{1,L_1^{l^*}\}}$.

Let
$$
t_0=\min\{\pi_{L_\mathbf{v}}(\mathbf{n})\;:\;\mathbf{n}=(n_1,\cdots,n_k)\in \mathbb{Z}^k \text{ with } 0<|n_i|\le|n^\diamond_i|, 1\le i\le k\}.
$$
Clearly, for any $0<t<t_0$, we have
$$
L^t_\mathbf{v}\cap \mathbb{Z}^k=\{l\mathbf{n}^\diamond:l\in \mathbf{Z}\}.
$$
Therefore, we have that $L_\mathbf{v}\in\mathbb{A}_1(\alpha, \Gamma)$ and hence completes the proof of this proposition.
\end{proof}

In the following we will consider the harder case.

 \begin{proposition}\label{case2}
Let $\alpha$ be a $C^r$ $\mathbb{Z}^k$-action on $M$ as in Theorem A. If $\mathbf{v}=(v_1,\cdots,v_k)\in\mathbb{R}^k$ is an irrational regular vector, then $L_\mathbf{v}\in\mathbb{A}_1(\alpha, \Gamma)$.
\end{proposition}

To prove this proposition, we will use the strategy similar to that for Proposition \ref{case1}. However, there is no non-zero $\mathbf{n}=(n_1,\cdots,n_k)\in\mathbb{Z}^k\cap L_\mathbf{v}$ in this case. We will choose a sequence of diffeomorphisms of $M$ carefully in a thin tube around $L_\mathbf{v}$ such that the induced nonautonomous system is topologically Anosov on $\Gamma$,  and then  deduce $L_\mathbf{v}\in\mathbb{A}_1(\alpha, \Gamma)$.

Now, we introduce some basic notations and establish some basic properties for nonautonomous dynamical systems induced by a sequence of maps. Let $\{g_p\}_{p\in\mathbb{Z}}$ be a sequence of diffeomorphisms of $M$. For any $p\in\mathbb{Z}, m\ge 0$, define
\begin{equation}\label{nonauto}
g_p^m:=\left\{\begin{array}{ll}
g_{p+m-1}\circ \cdots \circ g_p \quad &\;\;m\ge 2\\
g_p\quad &\;\;m=1\\
id \quad &\;\;m=0.
\end{array}\right.
\end{equation}
Denote  by $\mathcal{G}$ the nonautonomous system generated by $\{g_p\}_{p\in\mathbb{Z}}$ (or say by $\{g_p^m:p\in\mathbb{Z}, m\ge 0\}$).
Let $\Lambda\subset M$ be invariant with respect to each $g_p, p\in\mathbb{Z}$. We call it is  $\mathcal{G}$-\emph{invariant}.

\begin{definition}
We call that $\Lambda\subset M$ is a \emph{hyperbolic set} of $\mathcal{G}$ if it is $\mathcal{G}$-invariant and there exist an invariant splitting $T_{\Lambda}M=E^{s}\bigoplus E^{u}$ and a constant $0<\lambda<1$ such that for $p\in\mathbb{Z}, m\ge 0$,
$$
\|Dg_p^m(x)v\|\leq \lambda^m\|v\|\;\; \mbox{for}\;\; v\in E^s(x)
$$
and
$$
\|Dg_p^m(x)v\|\geq \lambda^{-m}\|v\| \;\; \mbox{for}\;\; v\in E^u(x).
$$
We call the number $\lambda$ a \emph{hyperbolicity constant} of $\mathcal{G}$
on $\Lambda$.
\end{definition}

Let $\varepsilon,\delta>0$. A sequence of points $\xi=\{x_{p}\}_{p\in\mathbb{Z}}$ is said to be a
$\delta$\emph{-pseudo orbit} for $\mathcal{G}$ on $\Gamma$ if $\xi\subset \Lambda$ and
$$
\sup_{p\in\mathbb{Z}}d(g_p(x_p),x_{p+1})\leq \delta.
$$
We say that $x\in M$ $\varepsilon$-\emph{shadows}  such a pseudo orbit $\xi$ if
$$
\max\{\sup_{p\ge 0}d(g_0^p(x),x_p),\;\sup_{p< 0}d((g_p^{-p})^{-1}(x),x_p)\}\leq \varepsilon.
$$

 We say that $\mathcal{G}$ \emph{has the Lipschitz shadowing property on $\Lambda$} provided there exist $\delta_0, L>0$ such that any  $\delta$-pseudo orbit for $\mathcal{G}$ in $\Lambda$ $\delta\le \delta_0$ can be $L\delta$-shadowed by some point $x\in M$. For any $x,y\in \Lambda$, let
$$
d_{\mathcal{G}}(x,y):=\max\{\sup_{p\ge 0}d(g_0^{p}(x),g_0^{p}(y)),\;\sup_{p< 0}d((g_p^{-p})^{-1}(x),(g_p^{-p})^{-1}(y))\}
$$
  We say that $\mathcal{G}$ is \emph{expansive} on $\Lambda$ provided there is an \emph{expansive constant} $\rho_{\Lambda}>0$ such
that for any $x,y\in \Lambda$,
$$
d_{\mathcal{G}}(x,y)\leq \rho_{\Lambda} \Longrightarrow x=y.
$$
We say that $\mathcal{G}$ \emph{is topologically Anosov on $\Lambda$} provided $\mathcal{G}$ is expansive and has the shadowing property on $\Lambda$.

 \begin{lemma}\label{mainlemma}
Let $\{g_p\}_{p\in\mathbb{Z}}$ be a sequence of diffeomorphisms of $M$, $\mathcal{G}$ be the induced nonautonomous system and $\Lambda$ be $\mathcal{G}$-invariant. If $\Lambda\subset M$ is hyperbolic for $\mathcal{G}$, then $\mathcal{G}$ has the Lipschitz shadowing property and is expansive,  and hence is topologically Anosov on $\Lambda$.
\end{lemma}

\begin{proof}
Assume $\Lambda\subset M$ is a  hyperbolic set for $\mathcal{G}$ with the invariant splitting $T_{\Lambda}M=E^{s}\bigoplus E^{u}$ and a hyperbolicity constant $\lambda$. We will adapt the standard techniques in proving that a diffeomorphism is topologically Anosov on a hyperbolic set (Lemma \ref{Wen}) to the nonautonomous case. We only give the outline of the proof.

\emph{Part 1}. We show that $\mathcal{G}$ has the Lipschitz shadowing property on $\Lambda$.
We will show that for a $\delta$(sufficiently small)-pseudo orbit $\xi=\{x_p\}_{p\in \mathbb{Z}}$ of
$\mathcal{G}$ on $\Lambda$,  one can find a (unique) point $x\in M$ and some $L>0$ such that $x$ $L\delta$-shadowing $\xi$, i.e.,
\begin{equation}\label{shadow1}
\max\{\sup_{p\ge 0}d(g_0^p(x),x_p),\;\sup_{p< 0}d((g_p^{-p})^{-1}(x),x_p)\}\leq L\delta.
\end{equation}
Once this is done, we let
\begin{equation}\label{v_p}
v_p^*:=\left\{\begin{array}{ll}
\exp_{x_p}^{-1}(g_0^p(x)) \quad &\;\;p\ge 0\\
\exp_{x_p}^{-1}((g_p^{-p})^{-1}(x))\quad &\;\;p< 0,
\end{array}\right.
\end{equation}
where $\exp_{x}$ is the standard exponential mapping from
a small neighborhood of the zero vector in $T_{x}M$ to its image,
and hence the sequence of the vectors $v^*:=\{v_p^*:p\in \mathbb{Z}\}$ satisfying
\begin{equation}\label{mainequa}
v_{p+1}^*=\exp_{x_{p+1}}^{-1}\circ g_p\circ\exp_{x_p}(v_p^*),\;\;p\in \mathbb{Z}.
\end{equation}

For such pseudo orbit $\xi=\{x_p\}_{p\in \mathbb{Z}}$, define a Banach space
$$
\mathfrak{X}=\{v=\{v_p\}_{p\in \mathbb{Z}}:v_p\in T_{x_p}M, p\in \mathbb{Z}\},
$$
with the norm $\|v\|=\sup_{p\in \mathbb{Z}}\|v_p\|$.
Define an operator $\Psi$ on a small neighborhood $\mathfrak{B}$ of the zero section in $\mathfrak{X}$ by
\begin{equation}\label{betaThmA}
(\Psi(v))_{p+1}=\exp_{x_{p+1}}^{-1}\circ g_p\circ \exp_{x_p}(v_p),
\end{equation}
and a linear operator $A$ on $\mathfrak{B}$ by
\begin{equation}\label{AThmA}
(Av)_{p+1}=((A^s+A^u)v)_{p+1}=(A_p^s+A_p^u)v_p,
\end{equation}
where
$$
A_p^s=\Pi_{x_{p+1}}^s\circ
D(\exp_{x_{p+1}}^{-1}\circ g_p\circ \exp_{x_p})(0)\circ \Pi_{x_p}^s
$$
and
$$
A_p^u=\Pi_{x_{p+1}}^u\circ
D(\exp_{x_{p+1}}^{-1}\circ g_p\circ \exp_{x_p})(0)\circ \Pi_{x_p}^u,
$$
in which $\Pi^{s}_x:T_{x}M\to E^{s}_x$ is the projection onto
$E^{s}_x$ along $E^u_x$ and  $\Pi^{u}_x$ is
defined in a similar way. Let $\Phi=\Psi-A$.
By (\ref{betaThmA}) and (\ref{AThmA}),
(\ref{mainequa}) is equivalent to
\begin{equation}\label{mainequa1}
v^*=\Psi(v^*)=Av^*+\Phi(v^*).
\end{equation}

From (\ref{angle}) and the fact $D(\exp_{x})(0)=\id$, for any $\lambda'\in(\lambda, 1)$ we can take $\delta$ small enough such that
$$
\max_{p\in \mathbb{Z}}\max\{\|A_p^s|_{E^s_{x_p}}\|,\|(A_p^u|_{E^u_{x_{p+1}}})^{-1}\|\}<\lambda',
$$
i.e., $\max\{\|A^s|_{E^s}\|,\|(A^u|_{E^u})^{-1}\|\}<\lambda'$.
Since $\mathcal{G}$ is smooth,  $\Phi$ is a Lipschitz map and its Lipschitz constant Lip$\Phi$ tends to zero as $\delta$ tends to zero.
By these properties about $A$ and $\Phi$, we can apply a standard technical lemma (see Lemma 2.5 of \cite{Wen}, for example) for the operator $\Psi=A+\Phi$ to conclude that there exist $\delta_0, L>0$ such that for any
$\delta(<\delta_0)$-pseudo orbit $\{x_p\}_{p\in \mathbb{Z}}$, the operator
$\Psi$ is well-defined on the $L\delta$-neighborhood $\mathfrak{B}$ of the zero section in $\mathfrak{X}$ and is contracting, and therefore has a fixed point $v^*$ in $\mathfrak{B}$, clearly $\|v^*\|\le L\delta$. We get the desired property.

\emph{Part 2}. We prove that $\mathcal{G}$ is expansive on $\Lambda$. Since $\mathcal{G}$ is hyperbolic on $\Lambda$, using the standard graph  transform method we can establish the stable manifold theorem for $\mathcal{G}$ on $\Lambda$. Moreover, there exists $r>0$ and $\rho_\Lambda>0$ such that for $x, y\in \Lambda$ with $d(x,y)\le \rho_\Lambda$ then
the local stable manifold $W^s_r(x)$ (resp. the local unstable manifold $W^u_r(x)$) and local unstable manifold $W^u_r(y)$ (resp. the local stable manifold $W^s_r(y)$) of size $r$ intersects transversally, here for $x\in \Lambda$,
$$
W^s_r(x):=\{y\in M:\sup_{p\ge 0}d(g_0^{p}(x),g_0^{p}(y))\le r \}
$$
and
$$
W^u_r(x):=\{y\in M:\sup_{p\le 0}d((g_p^{-p})^{-1}(x),(g_p^{-p})^{-1}(y))\le r\}.
$$
 Therefore, for $x,y\in \Lambda$ satisfying $d_{\mathcal{G}}(x,y)\leq \rho_{\Lambda}$, we have that $y\in W^s_r(x)\cap W^u_r(x)$.
From the fact $W^s_r(x)\cap W^u_r(x)=\{x\}$ for sufficiently small $r$ we get $x=y$ immediately.

This completes the proof of the lemma.
\end{proof}

 \begin{proof}[Proof of Proposition \ref{case2}]
Let $\mathbf{v}=(v_1,\cdots,v_k)$ be an irrational nonzero vector which satisfies (\ref{Regular}). For the similar reason as at the beginning of the proof of Proposition \ref{case1},  we assume that each component $v_i$ in $\mathbf{v}$ is not $0$.

Now we take some quantities as in the proof of Proposition \ref{case1}. Precisely, by (\ref{Regular}), we write the set $\{1,\cdots,s\}$ into a disjoint union $J_{1}\cup J_{2}$, where $J_1$ and $J_2$ are as in (\ref{J1J2}).
Take $a>0$ small enough such that the inequalities in (\ref{bv12}) hold and let $b_{\mathbf{v}}=\min\{-b_{\mathbf{v},1}, b_{\mathbf{v},2}\}$ as that in (\ref{bv}). By the basic assumption, for the above $a>0$ take $N=N(a)>0$ such that when $|n|\ge N$, the inequalities in (\ref{Uniformity}) hold.

Note that the vector $\mathbf{v}=(v_1,\cdots,v_k)$ is irrational and satisfies the condition (\ref{Regular}). We can take a positive number $t_0$ in $[\sqrt{k}, 2\sqrt{k})$ as the thickness of $L_\mathbf{v}$, and then choose the above $a>0$ small enough and the corresponding $N=N(a, t_0)$ large enough such that there exists a sequence $\{\mathbf{n}^{(p)}=(n^{(p)}_1,\cdots,n^{(p)}_k)\}_{p\in\mathbb{Z}}\in L^{t_0}_\mathbf{v}\cap \mathbb{Z}^k$ with $\mathbf{n}^{(0)}=\textbf{0}$ and $\text{sgn}(n^{(p+1)}_i-n^{(p)}_i)=\text{sgn}(v_i), 1\le i\le k$, and  satisfying the following conditions
\begin{equation}\label{N2N}
N\le \sup_{p\in\mathbb{Z}}\max_{1\le i\le k}|n^{(p+1)}_i-n^{(p)}_i|\le 2N,
\end{equation}
\begin{equation}\label{condition1}
\lambda_1:=\sup_{p\in\mathbb{Z}} \max_{j\in J_1}\sum_{i=1}^k(n^{(p+1)}_i-n^{(p)}_i)(\lambda_{i,j}+a)<0
\end{equation}
and
\begin{equation}\label{condition2}
\lambda_2:=\inf_{p\in\mathbb{Z}} \min_{j\in J_2} \sum_{i=1}^k(n^{(p+1)}_i-n^{(p)}_i)(\lambda_{i,j}-a)>0.
\end{equation}
Let
$$
\lambda=\min\{e^{\lambda_1}, e^{-\lambda_2}\}.
$$
Then $0<\lambda<1$.

We \emph{claim} that the nonautonomous system $\mathcal{G}$ generated by the sequence $\{g_p=\alpha^{\mathbf{n}^{(p+1)}-\mathbf{n}^{(p)}}\}_{p\in\mathbb{Z}}$ of diffeomorphisms on $M$ is  hyperbolic on $\Gamma$.  For any $p\in\mathbb{Z}, m\ge 0$, define $g_p^m$ as in (\ref{nonauto}).
For each $x\in \Gamma$, denote $E^s(x)$ and $E^u(x)$ as in (\ref{EsEu}).
From the choice of $\mathcal{G}$ we conclude that for each $0\neq v\in E^s(x)$ and $m>0$,
\begin{eqnarray*}
\|Dg_p^m(x)v\|&=&\|D(g_{m+p-1}\circ \cdots \circ g_p)(x)v\|\\
&\le&\prod_{q=0}^{m-1}e^{b_{\mathbf{n}^{(p+q+1)}-\mathbf{n}^{(p+q)},1}}\|v\|\le(e^{\lambda_1})^m\|v\|\le\lambda^m\|v\|.
\end{eqnarray*}
Similarly, for each $0\neq u\in E^u(x)$ and $m>0$,
\begin{eqnarray*}
\|Dg_p^m(x)u\|&=&\|D(g_{m+p-1}\circ \cdots \circ g_p)(x)u\|\\
&\ge&\prod_{q=0}^{m-1}e^{b_{\mathbf{n}^{(p+q+1)}-\mathbf{n}^{(p+q)},2}}\|u\|\ge(e^{-\lambda_2})^m\|u\|\ge\lambda^{-m}\|u\|.
\end{eqnarray*}
Thus the claim holds.

By the above claim and Lemma \ref{mainlemma},  $\mathcal{G}$ has the Lipschitz shadowing property (with the corresponding constants $\delta^*_0, L^*>0$) and is expansive (with the expansiveness constant $\rho^*>0$) on $\Gamma$. By (\ref{N2N}), there is an upper bound, say $2N\sqrt{k}$, of the set $\{|\mathbf{n}^{(p+1)}-\mathbf{n}^{(p)}| : p\in\mathbb{Z}\}$, and note that the thickness $t_0\in[\sqrt{k}, 2\sqrt{k})$ of $L_\mathbf{v}$ is finite (in fact $t_0$ is quite small with respect to $N$). From this observation, we can use a strategy similar to that in the proof of Proposition \ref{case1} to show that there exist $\delta_0, L, L_2>0$ satisfying the following properties: let $\xi=\{x_{\mathbf{n}}: \mathbf{n}\in L^{t_0}_\mathbf{v}\cap \mathbb{Z}^k\}$ be a $\delta$-pseudo orbit of $L^{t_0}_\mathbf{v}$ for $\alpha$ in $\Gamma$ with $\delta\le \delta_0$, then its subsequence $\xi'=\{x_{\mathbf{n}^{(p)}}: p\in\mathbb{Z}\}$ must be a $L_2\delta(<\delta_0^*)$-pseudo orbit for $\mathcal{G}$, and hence can be $L^*L_2\delta$-shadowed by some point $x\in M$ for $\mathcal{G}$, and therefore, $\xi$  can be $L\delta$-shadowed by $x$ with respect to $L^{t_0}_\mathbf{v}$ for $\alpha$.  Thus $L_\mathbf{v}$ has the Lipschitz shadowing property in $\Gamma$ with the corresponding constants $\delta_0, L>0$. Furthermore, take $\rho>0$ small enough such that for any $x,y\in M$,
$$
d_{\alpha}^{L^{t_0}_\mathbf{v}}(x, y)\le \rho\Longrightarrow d_{\mathcal{G}}(x,y)\le \rho^*,
$$
and hence $x=y$ since $\rho^*$ is an expansiveness constant of $\mathcal{G}$. This means that
 $L_\mathbf{v}$ is expansive on $\Gamma$ and $\rho$ is an expansiveness constant.

This completes the proof of this proposition.
\end{proof}

\begin{proof}[Proof of Theorem A]
Let $\alpha:\mathbb{Z}^k\longrightarrow$ Diff$^{r}(M, M)$ be a $C^r$ $\mathbb{Z}^k$-action on $M$ which satisfies the basic  assumption. Let  $\mathbf{v}=(v_1,\cdots,v_k)\in\mathbb{R}^k$ be regular. By Proposition \ref{case1} and Proposition \ref{case2}, $L_\mathbf{v}\in\mathbb{A}_1(\alpha, \Gamma)$. Hence by (2) of Proposition \ref{BP3}, we have that for each $1\le i\le k$,
$$
\mathbb{A}_i(\alpha, \Gamma)=\{V\in \mathbb{G}_i : V \text{ contains a regular vector} \}.
$$
\end{proof}

A classical property for a diffeomorphism on a hyperbolic set is \emph{structural stability} (see Theorem 4.19 of \cite{Wen}, for example). Precisely, let $g:M\longrightarrow M$ be a diffeomorphism and $\Lambda\subset M$ a hyperbolic set of $g$. Then for any diffeomorphism $g'$ which is $C^1$ close to $g$, then there exists a homeomorphism $h$ from $\Lambda$ onto its image such that $g'\circ h=h\circ g$ and $h$ is $C^0$ close to $id$. There are two classical methods to prove this result, one of them is to apply the shadowing property and expansivity of $g$ on $\Lambda$. Hence by Theorem A, we have the following property concerning stability for $\alpha$.

\begin{proposition}\label{stability}
Let $\alpha:\mathbb{Z}^k\longrightarrow$ Diff$^{r}(M, M)$ be a $C^r$ $\mathbb{Z}^k$-action on $M$ which satisfies the basic assumption. Let $\mathbf{v}\in \mathbb{Z}^k$ is a rational regular vector, then $\alpha$ is structurally stable along $L_{\mathbf{v}}$ in the following sense:
for any smooth $\mathbb{Z}^k$-action $\alpha'$ which is $C^1$ close to $\alpha$, then there exists a homeomorphism $h$ from $\Lambda$ onto its image such that $(\alpha')^{\mathbf{n}}\circ h=h\circ \alpha^{\mathbf{n}}$ for any $\mathbf{n}\in L_{\mathbf{v}}\cap \mathbb{Z}^k\setminus \{\mathbf{0}\}$, and moreover, $h$ is $C^0$ close to $id$.
\end{proposition}

We can see Theorem 8.1.18 of \cite{Katok1} for a similar stability result for smooth actions of abelian groups.

\subsection{\texorpdfstring{Quasi-shadowing for $\alpha$ along 1-dimensional subspace containing singular vectors (Proof of Theorem B)}{Quasi-shadowing for α along 1-dimensional subspace containing singular vectors (Proof of Theorem B)}}

In this subsection, we will consider the shadowing property for $\alpha$ along 1-dimensional subspaces of $\mathbb{R}^k$ which contains singular vectors. Let $\mathbf{v}\in\mathbb{R}^k$ be a singular vector, recall that we say $v$ is a first-type singular vector if there exist at least two indexes $j,j'\in \{1,\cdots,s\}$ such that
$
\sum_{i=1}^kv_i\lambda_{i,j}=0\text{ and }\sum_{i=1}^kv_i\lambda_{i,j'}\neq 0.
$
We say $v$ is a \emph{second-type} singular vector if for all $j\in \{1,\cdots,s\}$,
$
\sum_{i=1}^kv_i\lambda_{i,j}=0.
$

\begin{example}
For the ${\mathbb{Z}}^2$-action $\alpha$ on the torus $\mathbb{T}^2$ in Example \ref{example3}, let $L_1$ be the line in $\mathbb{G}_1$ with slope 1. Then any nonzero vector in $L_1$ is a second-type singular vector, and any $\mathbf{v}\in \mathbb{R}^2\setminus L_1$ is regular.

 For the ${\mathbb{Z}}^2$-action $\alpha$ on the torus $\mathbb{T}^3$ in Example \ref{example3.5}, let $L_1$ be the line in $\mathbb{G}_1$ with slope 1 as above. Then any nonzero vector in $L_1$ is a second-type singular vector, any $\mathbf{v}\in \mathbb{R}^2\setminus L_1$ is a first-type singular  vector, and hence there is no regular vector.

  Moreover, for the ${\mathbb{Z}}^2$-action $\alpha$ on the torus $\mathbb{T}^4$ in Example \ref{example4}, let $L_i$ be the line in ${\mathbb{R}}^2$ with slope $-\log |\lambda_i|= \log |\mu_1|, i=1, 2$. Then any nonzero vector in $L_1\cup L_2$ is a first-type singular vector, and any $\mathbf{v}\in \mathbb{R}^2\setminus (L_1\cup L_2)$ is regular.
\end{example}

Let  $\mathbf{v}=(v_1,\cdots,v_k)\in\mathbb{R}^k$ be a  first-type singular vector of $\alpha$,
by (\ref{1singular}), we can write the set $\{1,\cdots,s\}$ into a disjoint union of the following three subsets
\begin{equation}\label{j1j2j3}
J_{1}=\{j:\sum_{i=1}^kv_i\lambda_{i,j}<0\}, J_{2}=\{j:\sum_{i=1}^kv_i\lambda_{i,j}>0\}\;\;\text{ and } J_{3}=\{j:\sum_{i=1}^kv_i\lambda_{i,j}=0\}.
\end{equation}
For $x\in \Gamma$, define
\begin{equation}\label{EsEuEc}
E^s(x)=\bigoplus_{j\in J_{1}}E_j(x),\;E^u(x)=\bigoplus_{j\in J_{2}}E_j(x)\;\text{ and }\;E^c(x)=\bigoplus_{j\in J_{3}}E_j(x).
\end{equation}
Then we get an invariant splitting  $T_\Gamma M=E^s\bigoplus E^c \bigoplus E^u$.

Since there is a center distribution $E^c$ in the above splitting, we can not expect that $\alpha$ has the shadowing property along $L_\mathbf{v}$. However, we can consider the so-called quasi-shadowing property instead.

\begin{definition}\label{QS}
 Let $\mathbf{v}=(v_1,\cdots,v_k)\in\mathbb{R}^k$ be a  first-type singular vector of $\alpha$,  $T_\Gamma M=E^s\bigoplus E^c \bigoplus E^u$ be the splitting defined as above.
We say that $L_\mathbf{v}$ \emph{has the  quasi-shadowing property on $\Gamma$ for $\alpha$} provided  there exists  $t>0$ satisfying the following property: for any $\varepsilon>0$ there exists $\delta>0$ such that every $\delta$-pseudo orbit $\xi=\{x_{\mathbf{n}} : \mathbf{n}\in  L_\mathbf{v}^t\cap\mathbb{Z}^k\}$ of $L_\mathbf{v}^t$ for $\alpha$ in $\Gamma$, there is a set of points $\{y_{\mathbf{n}} : \mathbf{n}\in  L_\mathbf{v}^t\cap\mathbb{Z}^k\}$ $\varepsilon$-tracing it in which $y_{\mathbf{n'}}, \mathbf{n}'\in  L_\mathbf{v}^t\cap\mathbb{Z}^k$, is obtained from $\alpha^{\mathbf{n'}-\mathbf{n}}(y_{\mathbf{n}})$ by a motion $\tau$
along the center direction (please see that in Lemma \ref{mainlemma2}, for the meaning of  the motion $\tau$).
\end{definition}

 We can also see other shadowing properties which are analogous to the quasi-shadowing property in different settings in  \cite{Carrasco}, \cite{Tikhomirov} and \cite{Bonatti}, etc.

\begin{remark}
Let $\alpha$ be a smooth $\mathbb{Z}^k$-action on $M$  as in Theorem A and B.

(1)  From Theorem A and Theorem B we can see that once a subspace $V$ of $\mathbb{R}^k$ contains a regular vector, then $V\in A(\alpha, \Gamma)$ although there may be some (even many) singular vectors lying in $V$. This reflects that by the commutativity of the generators of $\alpha$ the shadowing property of $\alpha$ along $V$ determines by weather $V$ contains a regular direction.

(2) A natural question is: how about the 1-dimensional subspace $V$ which contains a  second-type singular vector? In some particular cases, such as ``linear"-actions on torus, $V$ does not have shadowing property.  For example, in Example \ref{example3}, $L_1$ has no shadowing property, where $L_1$ is the line in $\mathbb{G}_1$ with slope 1. However, we do not know weather there is an example which has the shadowing property along a 1-dimensional subspace $V$ which contains a  second-type singular vector.
\end{remark}

We will apply the similar method as in Theorem A to prove Theorem B. We need to introduce the property of quasi-shadowing property for both autonomous and nonautonomous partially hyperbolic diffeomorphisms.

Let $\{g_p\}_{p\in\mathbb{Z}}$ be a sequence of diffeomorphisms of $M$ and  $\mathcal{G}$ the (usually nonautonomous) system generated by $\{g_p\}_{p\in\mathbb{Z}}$. In particular, when $g_p=g$ for any $p\in\mathbb{Z}$, then $\mathcal{G}$ is the autonomous system generated by $g$.

\begin{definition}
Let $\mathcal{G}$ be the system generated by $\{g_p\}_{p\in\mathbb{Z}}$. We call that $\Lambda\subset M$  is a \emph{partially hyperbolic set} of $\mathcal{G}$ if it is $\mathcal{G}$-invariant and there exist an invariant splitting $T_{\Lambda}M=E^{s}\bigoplus E^{c}\bigoplus E^{u}$ and constants $0<\lambda<\mu<1$ such that for $p\in\mathbb{Z}, m\ge 0$,
$$
\|Dg_p^m(x)v\|\leq \lambda^m\|v\|\;\; \mbox{for}\;\; v\in E^s(x),
$$
$$
\mu^m\|v\|\leq\|Dg_p^m(x)v\|\leq \mu^{-m}\|v\|\;\; \mbox{for}\;\; v\in E^c(x)
$$
and
$$
\|Dg_p^m(x)v\|\geq \lambda^{-m}\|v\| \;\; \mbox{for}\;\; v\in E^u(x).
$$
We call the numbers $\lambda$ and $\mu$ the \emph{partial hyperbolicity constants} of $\mathcal{G}$
on $\Lambda$.
\end{definition}

 By a similar discussion as in Theorem 4.3 of \cite{Wen}, for the partially hyperbolic set $\Lambda$ of $\mathcal{G}$, the corresponding decomposition $T_{\Lambda}M=E^s\bigoplus E^{c}\bigoplus E^{u}$ is continuous, and moreover, it can be continuously extended to the closure  of $\Lambda$ such that $\mathcal{G}_{\overline{\Lambda}}$ is also partially hyperbolic. By this observation, we have that there exists a positive number $r$
 \begin{equation}\label{angle1}
 \sup_{x\in \Lambda} \max\{\angle(E^i(x),E^j(x)):i,j=s,c,u,i\neq j\}\ge r.
 \end{equation}

For a sequence of points $\{x_p\}_{p\in \mathbb{Z}}\subset \Lambda$
and a sequence of vectors $\{u_p\in E^c_{x_p}\}_{p\in \mathbb{Z}}$
with $\|u_p\|$ is sufficiently small for any $p\in \mathbb{Z}$,
we define a family of smooth maps $\tau_{x_p}=\tau_{x_p}(\cdot, u_p)$
on a small neighborhood of $x_p$, as in \cite{Hu2} and \cite{Hu1},  by
\begin{equation}\label{tau}
\tau_{x_p}(y)=\exp_{x_p}(u_p+\exp_{x_p}^{-1}y).
\end{equation}

 \begin{lemma}\label{mainlemma2}
Let $\{g_p\}_{p\in\mathbb{Z}}$ be a sequence of diffeomorphisms of $M$, $\mathcal{G}$ be the induced  system and $\Lambda$ be $\mathcal{G}$-invariant. If $\Lambda\subset M$ is partially hyperbolic for $\mathcal{G}$, then $\mathcal{G}$ has the
quasi-shadowing property on $\Lambda$ in the following sense: for any $\varepsilon>0$ there exists $\delta>0$ such that for any $\delta$-pseudo
orbit $\{x_p\}_{p\in \mathbb{Z}}$ of $\mathcal{G}$ in $\Lambda$, there exist a sequence of points $\{y_p\}_{p\in \mathbb{Z}}$ and
a sequence of vectors $\{u_p\in E^c_{x_p}\}_{p\in \mathbb{Z}}$
such that
\begin{equation}\label{eqn1thmA}
d(x_p,y_p)<\varepsilon,
\end{equation}
where
\begin{equation}\label{eqn2thmA}
y_p=\tau_{x_p}(g_{p-1}(x_{p-1})).
\end{equation}

Moreover, $\{y_p\}_{p\in \mathbb{Z}}$ and $\{u_p\}_{p\in \mathbb{Z}}$ can be chosen
uniquely so as to satisfy
\begin{equation}\label{eqn3thmA}
y_p\in \exp_{x_p}(E^s_{x_p}+E^u_{x_p}).
\end{equation}
\end{lemma}

\begin{proof}
Assume $\Lambda\subset M$ is a partially hyperbolic set for $\mathcal{G}$ with the invariant splitting $T_{\Lambda}M=E^{s}\bigoplus E^{c}\bigoplus E^{u}$ and the partial hyperbolicity constants $\lambda$ and $\mu$. We will adapt the techniques in proving that a partially hyperbolic diffeomorphism has the quasi-shadowing property (Theorem A of  \cite{Hu1}) to our case. We only give the outline of the proof, for the detailed estimation we refer to the proof of Theorem A in \cite{Hu1}.

Recall that $\|\cdot\|$ is the norm on $TM$. We define the norm
$\|\cdot\|_1$ on $TM$ by $\|w\|_1=\|u\|+\|v\|$ if $w=u+v\in T_xM$
with $u\in E^c_x$ and $v\in E^u_x\oplus E^s_x$. For any sequence $\{x_k\}_{k\in \mathbb{Z}}$ in $\lambda$, Denote
$$
\mathfrak{X}=\{w=\{w_p\}_{p\in \mathbb{Z}}:w_p\in T_{x_p}M, p\in \mathbb{Z}\},
$$
$$
\mathfrak{X}^c=\{u=\{u_p\}_{p\in \mathbb{Z}}:u_p\in E_{x_p}^c, p\in \mathbb{Z}\}
$$
and
$$
\mathfrak{X}^{us}=\{v=\{v_p\}_{p\in \mathbb{Z}}:v_p\in E_{x_p}^u\oplus E_{x_p}^s, p\in \mathbb{Z}\}.
$$
For any $w=u+v\in \mathfrak{X}$, where $u\in \mathfrak{X}^c$ and $v\in \mathfrak{X}^{us}$, we also define
$$
\|w\|=\sup_{p\in \mathbb{Z}}\|w_p\|
\text{ and }
\|w\|_1=\|u\|+\|v\|.
$$
By triangle inequality and the fact that the angles between $E^c$ and
$E^u\oplus E^s$ are uniformly bounded away from zero, i.e., (\ref{angle1}), we get that
$\|\cdot\|_1$ is equivalent to $\|\cdot\|$.

For any $\varepsilon>0$, we denote
$$
\mathfrak{B}(\varepsilon)=\{w\in \mathfrak{X}: \|w\|\le \varepsilon\},\qquad
\mathfrak{B}^{us}(\varepsilon)=\{w\in \mathfrak{X}^{us}: \|w\|\le \varepsilon\},
$$
$$
\mathfrak{B}_1(\varepsilon)=\{w\in \mathfrak{X}: \|w\|_1\le \varepsilon\}.
$$

We denote
$\Pi^{s}_x:T_{x}M\to E^{s}_x$ be the projection onto
$E^{s}_x$ along $E^c_x\oplus E^u_x$. $\Pi^{c}_x$ and $\Pi^{u}_x$ are
defined in a similar way.

Let $\{x_p\}_{p\in \mathbb{Z}}$ be a $\delta$-pseudo orbit  of
$\mathcal{G}$. To find a sequence of points $\{y_p\}_{p\in \mathbb{Z}}$ and
a sequence of vectors $\{u_p\in E^c_{x_p}\}_{p\in \mathbb{Z}}$
satisfying (\ref{eqn1thmA}), (\ref{eqn2thmA}) and (\ref{eqn3thmA}), we shall
try to solve the equations
\begin{equation}\label{feqn2ThmA}
y_p=\tau_{x_p}(g_{p-1}(x_{p-1})), \quad  p\in \mathbb{Z},
\end{equation}
for unknown $\{y_p\}_{p\in \mathbb{Z}}$ and $\{u_p\in E^c_{x_p}\}_{p\in \mathbb{Z}}$,
where $\tau_{x}$ is defined in (\ref{tau}).
Put $v_p=\exp_{x_p}^{-1}y_p, p\in\mathbb{Z}$.
Then the equations (\ref{feqn2ThmA}) can be written as
$$
v_p=\exp_{x_p}^{-1}\circ \tau_{x_p}( g_{p-1}\circ \exp_{x_{p-1}}v_{p-1}),
\quad  p\in \mathbb{Z}.
$$
By (\ref{tau}), the equations are equivalent to
\begin{equation}\label{eqn4ThmA}
v_p=u_p+\exp_{x_p}^{-1}\circ g_{p-1}\circ \exp_{x_{p-1}}v_{p-1},\quad p\in \mathbb{Z}.
\end{equation}

For sufficiently small $a>0$ define an operator $\beta: \mathfrak{B}^{us}(a)\to \mathfrak{X}$
and a linear operator $A: \mathfrak{B}^{us}(a)\to \mathfrak{X}^{us}$ by
\begin{equation}\label{beta4ThmA}
(\beta(v))_p=\exp_{x_p}^{-1}\circ g_{p-1}\circ \exp_{x_{p-1}}v_{p-1}
\end{equation}
and
\begin{equation}\label{A4ThmA}
(Av)_p=((A^s+A^u)v)_p=(A_{p-1}^s+A_{p-1}^u)v_{p-1},
\end{equation}
where
\begin{equation}
\begin{split}
A_{p-1}^s=\Pi_{x_p}^s\circ
D(\exp_{x_p}^{-1}\circ g_{p-1}\circ \exp_{x_{p-1}})(0)\circ \Pi_{x_{p-1}}^s, \\
A_{p-1}^u=\Pi_{x_p}^u\circ
D(\exp_{x_p}^{-1}\circ g_{p-1}\circ \exp_{x_{p-1}})(0)\circ \Pi_{x_{p-1}}^u.
\end{split}
\end{equation}

Let $\eta=\beta-A$.  By (\ref{beta4ThmA}) and (\ref{A4ThmA}),
(\ref{eqn4ThmA}) is equivalent to
$$
v=u+Av+\eta(v),
$$
or
$$
v-u-Av=\eta(v).
$$

Define a linear operator $P$ from a neighborhood of $0\in
\mathfrak{X}$ to $\mathfrak{X}$ by
\begin{equation}\label{PThmA}
Pw=-u+(\id_{\mathfrak{X}^{us}}-A)v,
\end{equation}
and then define an operator $\Phi$ from a neighborhood of $0\in
\mathfrak{X}$ to $\mathfrak{X}$ by
$$
\Phi(w)=P^{-1}\eta(v)
$$
for $w=u+v$ in the neighborhood of $0\in \mathfrak{X}$, where
$u\in \mathfrak{X}^c$ and $v\in \mathfrak{X}^{us}$.

Hence, (\ref{eqn4ThmA}) is equivalent to
\begin{equation}\label{feqn4ThmA}
\Phi(w)=w,
\end{equation}
namely, $w$ is a fixed point of $\Phi$.

Adapting the estimation in Lemma 3.1 of \cite{Hu1} to our case, we get that for any small $\varepsilon>0$ there exists $\delta=\delta(\varepsilon)$ such that for any
$\delta$-pseudo orbit $\{x_k\}_{k\in \mathbb{Z}}$, the operator
$\Phi: \mathfrak{B}_1(\varepsilon)\to \mathfrak{B}_1(\varepsilon)$ defined as above
is a contracting map, and therefore has a fixed
point in $\mathfrak{B}_1(\varepsilon)$.  Hence, (\ref{eqn4ThmA}) has a
unique solution.

This completes the proof of the lemma.
\end{proof}

 \begin{proof}[Proof of Theorem B]
Let $\mathbf{v}=(v_1,\cdots,v_k)\in\mathbb{R}^k$ be a first-type singular vector, i.e., there exist at least two indexes $j,j'\in \{1,\cdots,s\}$ such that (\ref{1singular}) holds. As before, we assume that each component $v_i$ in $\mathbf{v}$ is not $0$. We give the proof in two cases. Since the proof is an adaption from shadowing for regular case to quasi-shadowing for the first-type singular case, we only give the outline of the proof and leave the details to the readers.

\emph{Case 1}. $\mathbf{v}$ is rational.

By (\ref{1singular}), we can write the set $\{1,\cdots,s\}$ into a disjoint union of the three subsets $J_{1}, J_{2}$ and $J_{3}$ as in (\ref{j1j2j3}) and let $T_\Gamma M=E^s\bigoplus E^c \bigoplus E^u$ be the corresponding splitting defined by (\ref{EsEuEc}).
Let
$$
b_{\mathbf{v},1}=\max_{j\in J_1}\sum_{i=1}^kv_i(\lambda_{i,j}+ a), \;\;b_{\mathbf{v},2}=\min_{j\in J_2}\sum_{i=1}^kv_i(\lambda_{i,j}-a)
$$
and
$$
b_{\mathbf{v},3}=\max_{j\in J_3}\max\{\sum_{i=1}^kv_i(\lambda_{i,j}+ a), |\sum_{i=1}^kv_i(\lambda_{i,j}- a)|\}.
$$
Take $a>0$ small enough such that
\begin{equation}\label{b1b2b3}
b_{\mathbf{v},1}<0, b_{\mathbf{v},2}>0 \text{ and } b_{\mathbf{v},3}<\max\{b_{\mathbf{v},2},-b_{\mathbf{v},1}\}.
\end{equation}
Observe that, the gaps between any two of the three numbers $b_{\mathbf{lv},1}, b_{\mathbf{lv},2}$ and $b_{\mathbf{lv},3}$ increases while as $l>0$ increases.

By the basic assumption, for the above $a>0$ take $N>0$ such that when $|n|\ge N$, the inequalities in (\ref{Uniformity}) hold.
We take a nonzero vector $\mathbf{n}^*=(n^*_1,\cdots,n^*_k)\in L_\mathbf{v}\cap\mathbb{Z}^k$ such that $|n^*_i|>N$ for each $i=1,\cdots,k$.

We deduce that for the above $\mathbf{n}^*$, the diffeomorphism $g:=\alpha^{\mathbf{n}^*}$ on $M$ is partially hyperbolic on $\Gamma$ with the invariant splitting
$T_\Gamma M=E^s\bigoplus E^c \bigoplus E^u$
and the partial hyperbolicity constants $\lambda=e^{-b_{\mathbf{n}^*}}$, in which $b_{\mathbf{n}^*}=\min\{-b_{\mathbf{n}^*,1}, b_{\mathbf{n}^*,2}\}$, and $\mu=e^{-b_{\mathbf{n}^*,3}}$.

By Lemma \ref{mainlemma2}, we have that the above $\alpha^{\mathbf{n}^*}$ has the quasi-shadowing property.
Since the generators $f_i, 1\le j\le k$, of $\alpha$ are equi-continuous, we get that the diffeomorphism $\alpha^{\mathbf{n}^\diamond}$, where $\mathbf{n}^\diamond=(n^\diamond_1,\cdots,n^\diamond_k)$ is one of the two elements in  $L_\mathbf{v}\cap \mathbb{Z}^k\setminus \{\mathbf{0}\}$ with the smallest norm as in Proposition \ref{case1}, and  hence $L_\mathbf{v}$  has the quasi-shadowing property on $\Gamma$.

\emph{Case 2}. $\mathbf{v}$ is irrational.

Now we take some quantities as in Case1. Precisely, by (\ref{1singular}), we write the set $\{1,\cdots,s\}$ into a disjoint union $J_{1}\cup J_{2}\cup J_{3}$, where $J_1$ and $J_2$ are as in (\ref{j1j2j3}).
Take $a>0$ small enough such that the inequalities in (\ref{b1b2b3}) hold. By the basic assumption, for the above $a>0$ take $N=N(a)>0$ such that when $|n|\ge N$, the inequalities in (\ref{Uniformity}) hold.

Note that the vector $\mathbf{v}=(v_1,\cdots,v_k)$ is irrational and satisfies the condition (\ref{1singular}). We can take a positive number $t_0$ in $[\sqrt{k}, 2\sqrt{k})$ as the thickness of $L_\mathbf{v}$, and then choose the above $a>0$ small enough and the corresponding $N=N(a, t_0)$ large enough such that there exists a sequence $\{\mathbf{n}^{(p)}=(n^{(p)}_1,\cdots,n^{(p)}_k)\}_{p\in\mathbb{Z}}\in L^{t_0}_\mathbf{v}\cap \mathbb{Z}^k$ satisfying (\ref{N2N}), (\ref{condition1}), (\ref{condition2}) and the following inequality
$$
\lambda_3\le\min\{-\lambda_1,\lambda_2\},
$$
in which $\lambda_3$ is defined as the following positive number
$$
\sup_{p\in\mathbb{Z}} \max_{j\in J_3}\{ |\sum_{i=1}^k(n^{(p+1)}_i-n^{(p)}_i)(\lambda_{i,j}-a)|,\sum_{i=1}^k(n^{(p+1)}_i-n^{(p)}_i)(\lambda_{i,j}-a)\}.
$$

We get that the nonautonomous system $\mathcal{G}$ generated by the sequence $\{g_p=\alpha^{\mathbf{n}^{(p+1)}-\mathbf{n}^{(p)}}\}_{p\in\mathbb{Z}}$ of diffeomorphisms on $M$ is  partially hyperbolic on $\Gamma$ with the invariant splitting (\ref{EsEuEc})
and the partial hyperbolicity constants $\lambda=e^{-\min\{-\lambda_1, \lambda_2\}}$ and $\mu=e^{-\lambda_3}$.

By Lemma \ref{mainlemma2}, we have that the above $\mathcal{G}$ has the quasi-shadowing property.
By the equicontinuity of the generators $f_i, 1\le j\le k$, of $\alpha$, we get that $L_\mathbf{v}$  has the quasi-shadowing property on $\Gamma$.

This completes the proof of this theorem.
\end{proof}

An important property for a partially hyperbolic diffeomorphism is \emph{quasi-stability} (see \cite{Hu2}, for example). Precisely, let $g:M\longrightarrow M$ be a partially hyperbolic diffeomorphism. Then for any homeomorphism $g'$ which is $C^0$ close to $g$, then there
 exist a continuous map $h$ and a set of locally defined homeomorphisms $\{l_x:x\in M\}$ on $M$ ($l_x$ is a \emph{motion} along the center direction as in \cite{Hu1}) such that
$$
h\circ g=l_{f(\cdot)}\circ f\circ h,
$$
in which $h$ is $C^0$ close to $id$. Once the center foliation exists and is dynamically coherent with the stable and unstable foliations, then we can get the quasi-stability via quasi-shadowing property (\cite{Hu3}).
Here, we give the following property concerning quasi-stability for $\alpha$.

\begin{proposition}\label{stability2}
Let $\alpha:\mathbb{Z}^k\longrightarrow$ Diff$^{r}(M, M)$ be a $C^r$ $\mathbb{Z}^k$-action on $M$ which satisfies the basic assumption and $\Gamma=M$. Let $\mathbf{v}\in \mathbb{Z}^k$ be a first-type rational singular vector, then $\alpha$ is quasi-stable along $L_{\mathbf{v}}$ in the following sense:
for any smooth $\mathbb{Z}^k$-action $\alpha'$ which is $C^0$ close to $g$, then there
 exist a continuous map $h$ and a set of locally defined homeomorphisms $\{l_x:x\in M\}$ on $M$ such that
$$
h\circ (\alpha')^{\mathbf{n}}=l_{\alpha^{\mathbf{n}}(\cdot)}\circ \alpha^{\mathbf{n}}\circ h,
$$
for any $\mathbf{n}\in L_{\mathbf{v}}\cap \mathbb{Z}^k\setminus \{\mathbf{0}\}$, and moreover, $h$ is $C^0$ close to $id$.
\end{proposition}

\section{\texorpdfstring{A particular $\mathbb{R}^k$-action: suspension of a $\mathbb{Z}^k$-action}{A particular R\^{}k-action: suspension of a Z\^{}k-action}}

In previous sections, we investigated the shadowing property for subsystems of $\mathbb{Z}^k$-actions. Since a subspace $V$ of $\mathbb{R}^k$ may be ``invisible"
within $\mathbb{Z}^k$, we use the technique of thickening $V$ by a positive number $t$ to make $V^t$ is ``visible" within $\mathbb{Z}^k$. An alternative technique is to take the  suspension of a $\mathbb{Z}^k$-action, as in \cite{Kalinin1} for example.

Generally, the corresponding notions, such as  Lyapunov exponents, Lyapunov spectrum, Lyapunov hyperplanes and Weyl chambers, for a smooth $\mathbb{R}^k$ action are defined similar to that for smooth $\mathbb{Z}^k$ actions. We note that any $\mathbb{R}^k$ action has $k$ identically zero Lyapunov exponents corresponding to the
orbit directions. These Lyapunov exponents are called \emph{trivial} and the other
ones are called \emph{nontrivial}. When calling a Lyapunov exponent of
an $\mathbb{R}^k$ action we usually mean a nontrivial one.

For a smooth $\mathbb{Z}^k$-action $\alpha$ on $M$, there  is a natural $\mathbb{R}^k$ action on the suspension manifold $S$, which is a bundle
over $\mathbb{T}^k$ with fiber $M$. Namely, let $\mathbb{Z}^k$ act on $\mathbb{R}^k\times M$ by $\mathbf{n}(\mathbf{u}, x)=(\mathbf{u}-\mathbf{n}, \alpha^{\mathbf{n}}(x))$ and form the quotient space
$$
S = \mathbb{R}^k\times M/\mathbb{Z}^k
$$
Note that the action $\beta$ of $\mathbb{R}^k$ on $\mathbb{R}^k\times M$ by $\mathbf{n}(\mathbf{u}, x)=(\mathbf{n}+\mathbf{u}, x)$ commutes with the
$\mathbb{Z}^k$-action and therefore descends to $S$. This $\mathbb{R}^k$-action is called the \emph{suspension}
of the $\mathbb{Z}^k$-action and is denoted by  $\tilde{\alpha}$. There is a natural correspondence between the invariant
measures, nontrivial Lyapunov exponents, Lyapunov distributions, stable and
unstable manifolds, etc. for  $\alpha$ and  $\tilde{\alpha}$.

 Let $\mu$ be an ergodic probability measure for $\alpha$ and $\Gamma$ the Oseledec set.  We have the following property for $\tilde{\alpha}$:
for each $\mathbf{v}=(v_1,\cdots,v_k)\in \mathbb{R}^k$ and each
 $1\le j\le s(x)$,
\begin{equation}\label{Lya3}
\lim_{t\longrightarrow \pm \infty}\frac{1}{t}\log\|\frac{\partial\tilde{\alpha}^{t\mathbf{v}}(\mathbf{u},x)}{\partial x}v\|=\sum_{i=1}^kv_i\lambda_{i,j}(x)
\end{equation}
for any $(\mathbf{u},x)\in S, x\in \Gamma$ and $ 0\neq v\in E_j(x)$, where $\pi_M$ is the projection from $S$ to $M$.

Now we can define the shadowing property and expansiveness for $\tilde{\alpha}$ on $\Gamma$ along a subspace $V$ of $\mathbb{R}^k$  as follows. Put
$$
d_{\tilde{\alpha}}^V (x, y) = \sup\{d(\pi_M\tilde{\alpha}^{\mathbf{v}}(\mathbf{0},x),\pi_M\tilde{\alpha}^{\mathbf{v}}(\mathbf{0},y)) : \mathbf{v}\in  V \},\;x,y\in M.
$$

\begin{definition}
We say that a subspace $V$ of $\mathbb{R}^k$ is \emph{expansive} on $\Gamma$ for $\tilde{\alpha}$ if there is an \emph{expansive constant} $\rho_V>0$ such that for any $x,y\in \Gamma$, $d_{\tilde{\alpha}}^V(x, y)\le \rho_V$ implies that $x=y$.  In particular, when $V=\mathbb{R}^k$, we say that $\tilde{\alpha}$ is \emph{expansive} on $\Gamma$.
\end{definition}

To give a reasonable definition of pseudo orbit for the subsystems of $\tilde{\alpha}$, we focus on 1-dimensional subspace of $\mathbb{R}^k$. Let $V$ be a 1-dimensional subspace of $\mathbb{R}^k$, then the restriction of $\tilde{\alpha}$ to $V$ naturally induce a flow on $S$, we denote it by $\tilde{\alpha}_V$. Let $\delta,a>0$. We say that a sequences $\xi=\{(\mathbf{u}_p, x_p), \mathbf{v}_p\}_{p\in \mathbb{Z}}$ is a \emph{$(\delta,a)$-chain along $V$} for $\tilde{\alpha}$  on $\Gamma$  if $(\mathbf{u}_p, x_p)\in S$ with $x_p\in \Gamma$, $\mathbf{v}_p\in V$ with $|\mathbf{v}_p|>a$ and
$$
\sup_{p\in \mathbb{Z}}d(\pi_M\tilde{\alpha}^{\mathbf{v}_p}(\mathbf{u}_p, x_p),  x_{p+1}) <\delta.
$$
Let $\varepsilon>0$, a point $x\in M$ is called $\varepsilon$-\emph{shadows} the above pseudo orbit $\xi$ if
$$
\sup_{p\in \mathbb{Z}}d(x_p, \pi_M\tilde{\alpha}^{\sum_{i=0}^{p-1}\mathbf{v}_j}(\mathbf{0}, x))\le \varepsilon.
$$

\begin{definition}
Let $V$ be a 1-dimensional subspace of $\mathbb{R}^k$.
 We say that $\tilde{\alpha}$ \emph{has the shadowing property along $V$ on $\Gamma$} (or say, the flow $\tilde{\alpha}_V$ \emph{has the shadowing property
 on $\Gamma$}) provided for any $\varepsilon>0$ there exists $\delta>0$ such that for any $a>0$ and  every $(\delta,a)$-chain  $\xi=\{(\mathbf{u}_p, x_p), \mathbf{v}_p\}_{p\in \mathbb{Z}}$ along $V$ for $\tilde{\alpha}$  on $\Gamma$, there exists some point $x\in X$ $\varepsilon$-shadowing it. We will use the notion ``quasi-shadowing property" for $\tilde{\alpha}$, which is analogous to that in Lemma \ref{mainlemma2}.
\end{definition}

 \begin{proof}[Proof of Theorem C]
 By Theorem A and Theorem B, we get  Theorem C immediately.
\end{proof}

\section*{Acknowledgements}

The authors would like to thank the referees for the detailed review and very valuable suggestions which led to improvements of the paper.

\end{document}